\documentclass[11pt,reqno]{amsart}

\usepackage{paralist, amsmath, amsthm, amssymb, color, hyperref}

\usepackage[margin=1in]{geometry}

\usepackage{graphicx}
\DeclareGraphicsExtensions{.jpg,.pdf,.eps}
\graphicspath{{./Figures/}}

\usepackage{mathtools}

\usepackage{dsfont}
\usepackage{tikz}

\theoremstyle{plain}
\newtheorem{thm}{Theorem}[section] 
\newtheorem{lem}[thm]{Lemma}
\newtheorem{prop}[thm]{Proposition}
\newtheorem{cor}[thm]{Corollary}
\newtheorem{lemma}[thm]{Lemma}

\theoremstyle{definition}
\newtheorem{defn}[thm]{Definition} 
\newtheorem{definition}[thm]{Definition} 
\newtheorem{exmp}[thm]{Example} 
\newtheorem{example}[thm]{Example}
\newtheorem{remark}[thm]{Remark}
\newtheorem{question}[thm]{Question}

\newcommand{\ZZ}{\mathbb{Z}}
\newcommand{\RR}{\mathbb{R}}
\newcommand{\NN}{\mathbb{N}}
\newcommand{\PP}{\mathbb{P}}

\newcommand{\mL}{\mathcal{L}}
\newcommand{\mP}{\mathcal{P}}
\newcommand{\mX}{\mathcal{X}}

\newcommand{\wA}{\widetilde{A}}
\newcommand{\wD}{\widetilde{D}}

\newcommand{\al}{\alpha}
\newcommand{\be}{\beta}
\newcommand{\ga}{\gamma}
\newcommand{\de}{\delta}
\newcommand{\om}{\omega}
\newcommand{\si}{\sigma}
\newcommand{\De}{\Delta}

\usepackage{mathtools}
\newcommand{\orient}[1]{\overleftrightarrow #1}
\renewcommand{\l}{\left}
\renewcommand{\r}{\right}

\newcommand{\Potts}{P}

\newcommand{\B}{B}
\newcommand{\Q}{B}

\newcommand{\As}{>} 
\newcommand{\Ds}{<} 
\newcommand{\wAs}{\geq}
\newcommand{\wDs}{\leq}

\newcommand{\defeq}{\vcentcolon=}
\newcommand{\chrom}{\chi}
\newcommand{\tf}{{\widetilde{f}}}
\newcommand{\bDe}{{\overline{\De}}}
\newcommand{\uD}{{\underline{D}}}
\newcommand{\bI}{{\overline{I}}}

\newcommand{\wchrom}{\widetilde{\chrom}}
\newcommand{\fS}{\mathfrak{S}}
\newcommand{\partialchrom}{\chrom}

\DeclareMathOperator{\xx}{\mathbf{x}}

\DeclareMathOperator{\ext}{ext}
\DeclareMathOperator{\acyc}{acyc}
\DeclareMathOperator{\ori}{Orient}
\DeclareMathOperator{\Des}{Des}
\DeclareMathOperator{\Asc}{Asc}
\DeclareMathOperator{\asc}{asc}
\DeclareMathOperator{\des}{des}
\DeclareMathOperator{\inv}{inv}
\DeclareMathOperator{\maj}{maj}

\DeclareMathOperator{\vv}{\textrm{v}}
\DeclareMathOperator{\comp}{\textrm{c}}
\DeclareMathOperator{\Surj}{Surj}
\DeclareMathOperator{\ONE}{\mathds{1}}
\DeclareMathOperator{\Id}{Id}
\DeclareMathOperator{\QSym}{QSym}
\DeclareMathOperator{\Sym}{Sym}
\DeclareMathOperator{\out}{outdeg}
\DeclareMathOperator{\ind}{indeg}

\newcommand{\ov}{\overline}
\newcommand{\und}{\underline}
\newcommand{\ds}{\displaystyle}

\newcommand{\del}{\setminus}
\newcommand{\bu}{\bullet}

\newcommand{\fig}[3]{\begin{figure}[h!]\begin{center}\includegraphics[#1]{#2}\end{center}\caption{#3}\label{fig:#2}\end{figure}}

\title{Tutte polynomials for directed graphs}

\author{Jordan Awan and Olivier Bernardi}
\thanks{J.A. acknowledges support from a GAANN fellowship by the U.S. Department of Education. O.B. acknowledges support from NSF grant DMS-1400859, and the hospitality of the Mathematics department at MIT during the Spring of 2016. }
\date{\today}

\begin{document}
\setcounter{tocdepth}{2}

\begin{abstract}
The Tutte polynomial is a fundamental invariant of graphs. In this article, we define and study a generalization of the Tutte polynomial for directed graphs, that we name the \emph{$B$-polynomial}. The $B$-polynomial has three variables, but when specialized to the case of graphs (that is, digraphs where arcs come in pairs with opposite directions), one of the variables becomes redundant and the $B$-polynomial is equivalent to the Tutte polynomial.
We explore various properties, expansions, specializations, and generalizations of the $B$-polynomial, and try to answer the following questions:
\begin{itemize}
\item what properties of the digraph can be detected from its $B$-polynomial (acyclicity, length of directed paths, number of strongly connected components, etc.)?
\item which of the marvelous properties of the Tutte polynomial carry over to the directed graph setting?
\end{itemize}
The $B$-polynomial generalizes the strict chromatic polynomial of mixed graphs introduced by Beck, Bogart and Pham. We also consider a quasisymmetric function version of the $B$-polynomial which simultaneously generalizes the Tutte symmetric function of Stanley and the quasisymmetric chromatic function of Shareshian and Wachs.
\end{abstract}


\maketitle
\section{Introduction}
The Tutte polynomial is a fundamental invariant of graphs. There is a vast and rich literature about the Tutte polynomial; see for instance~\cite{Brylawski:Tutte-poly,Welsh:PottsTutte} or~\cite[Chapter 10]{Bollobas:Tutte-poly} for an introduction. In this article, we investigate a generalization of the Tutte polynomial to directed graphs (or \emph{digraphs} for short). Note that graphs are a special case of digraphs: they are the digraphs such that arcs come in pairs with opposite directions. 

Several digraph analogues of the Tutte polynomials have been considered in the literature. This includes in particular the \emph{cover polynomial} of Chung and Graham~\cite{Chung:cover-polynomial}, and the \emph{Gordon-Traldi polynomials}~\cite{Gordon:polynomials-digraphs}.
See~\cite{Chow:Tutte-poly-digraph} and references therein for an overview of these digraph invariants.
These digraph invariants do share some of the features of the Tutte polynomial. However, most are not proper generalizations of the Tutte polynomial, as they are not equivalent to the Tutte polynomial for the special case of graphs. The only exception is the Gordon-Traldi polynomial denoted by $f_8$ in~\cite{Gordon:polynomials-digraphs}, which is an invariant for \emph{vertex-ordered digraphs} (pairs made of a digraph and a linear order of its vertices).

In the present article, we define a new digraph invariant that we call the \emph{$B$-polynomial}. For any digraph $D$, this invariant is a polynomial in three variables denoted by $B_D(q,y,z)$. The $B$-polynomial generalizes the Tutte polynomial to all digraphs. Precisely, when the digraph $D$ corresponds to a graph, then $B_D(q,y,z)$ is \emph{equivalent} (that is, equal up to a change of variables) to the Tutte polynomial $T_D(x,y)$. 
In effect, the third variable $z$ becomes redundant in the case of graphs, but it is not for general digraphs. There are actually two additional relations between the $B$-polynomial invariant and the Tutte polynomial. 
First, for any graph $G$, $T_G(x,y)$ is equivalent to the average of $B_D(q,y,z)$ over all digraphs $D$ obtained by orienting $G$.
Second, for any digraph $D$, $B_D(q,y,y)$ is equivalent to the Tutte polynomial $T_G(x,y)$ of the \emph{graph $G$ underlying $D$} (that is, the graph obtained by forgetting the direction of the arcs). These three properties make the $B$-polynomial a legitimate generalization of the Tutte polynomial.

Without further ado, let us define the $B$-polynomial. For a digraph $D=(V,A)$, we define $B_D(q,y,z)$ as the unique polynomial in the variables $q,y,z$ such that for any positive integer~$q$,
\begin{equation}\label{eq:def-B}
B_D(q,y,z)=\sum_{f:V\to \{1,2,\ldots,q\}}y^{\#(u,v)\in A,~f(v)>f(u)}\,z^{\#(u,v)\in A,~f(v)<f(u)}.
\end{equation}
In words, the $B$-polynomial counts the \emph{$q$-colorings} of $D$ (arbitrary functions from $V$ to $\{1,2,\ldots,q\}$) according to the number of strict ascents and strict descents.

\begin{example}\label{exp:B-intro}
For the digraphs represented in Figure~\ref{fig:exp-intro}, one gets 
\begin{eqnarray*}
 B_D(q,y,z)&=&q+\frac{q(q-1)}{2}(y+z),\\
 B_{D'}(q,y,z)&=&q+q(q-1)(y^2+z^2+yz)+\frac{q(q-1)(q-2)}{6}(y^3+z^3+2yz(y+z)),\\
 B_{D''}(q,y,z)&=&q+q(q-1)yz(y+z+1)+\frac{q(q-1)(q-2)}{6}yz(y^2+z^2+4yz).
\end{eqnarray*}
\end{example}

\fig{width=.8\linewidth}{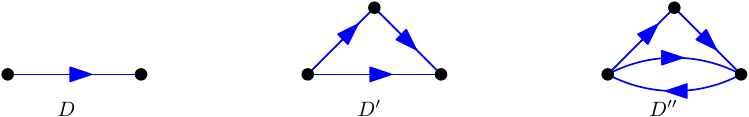}{Three digraphs.}

As we show in Section~\ref{sec:basics}, the three relations mentioned above between the $B$-polynomial and the Tutte polynomial actually follow pretty easily from the known relation between the Tutte polynomial and the Potts model. Indeed it is known that for any graph $G$, the Tutte polynomial $T_G(x,y)$ defined by~\eqref{eq:def-Tutte} is equivalent to the Potts polynomial $P_G(q,y)$ defined by~\eqref{eq:def-Potts}, via the change of variables~\eqref{eq:Potts=Tutte}. The relations between the $B$-polynomial and the Tutte polynomial then take the following form.
\begin{thm} \label{thm:Bpoly-Tutte}
For any undirected graph $G$, 
\begin{equation}\label{eq:PottsOne}
\Q_{\orient{G}}(q,y,z)=P_G(q,yz),
\end{equation}
where $\orient{G}$ is the digraph corresponding to $G$, and 
\begin{equation}\label{eq:PottsTwo}
\frac{1}{2^{|E|}}\sum_{\vec{G}\in \ori(G)} \Q_{\vec{G}}(q,y,z)
= P_G\left(q,\frac{y+z}{2}\right),
\end{equation}
where the sum is over all digraphs $\vec{G}$ obtained by orienting $G$.

Moreover, for any digraph $D$, 
\begin{equation}\label{eq:PottsThree}
\Q_D(q,y,y)=P_{\und{D}}(q,y),
\end{equation}
where $\und{D}$ is the graph underlying $D$.
\end{thm}

Our goal is to investigate how much of the theory of the Tutte polynomial extends to the digraph setting, and unearth some new identities. 
Some highlights are the following.
\begin{itemize}
\item The invariant $B_D$ detects whether $D$ is acyclic. More generally, $B_D$ contains the generating function of acyclic subgraphs of $D$, counted according to their number of arcs (see Corollary~\ref{cor:GFacyc2}). This generalizes to digraphs a result of Gessel and Sagan about acyclic sub-orientations of graphs~\cite{Gessel:Tutte-poly+DFS}. 
Additionally, $B_D$ contains the generating function of acyclic reorientations of $D$, counted according to their number of reoriented arcs (see Corollary~\ref{cor:GFacyc1}). 
\item The invariant $B_D$ detects whether $D$ is totally cyclic. 
More generally, $B_D$ contains the generating function of totally cyclic contractions of $D$ (see Corollary~\ref{cor:GFcyc2}) and the generating function of totally cyclic reorientations of $D$ (see Corollary~\ref{cor:GFcyc1}).
\item The invariant $B_D$ detects the length of the longest directed path in $D$, the number of strongly connected components, and the number of linear extensions of the acyclic graph obtained from $D$ by contracting all the strongly connected components (see Corollary~\ref{cor:additional-prop}).
\item The $B$-polynomial satisfies a partial planar duality relation (see Theorem~\ref{thm:duality}). 
\item Digraphs can also be seen as \emph{mixed graphs} (a.k.a. \emph{partially oriented graphs}) by interpreting the pairs of arcs in opposite direction as \emph{unoriented edges}. An invariant of mixed graphs called the \emph{(strict) chromatic polynomial of mixed graphs} was investigated in~\cite{Sotskov:ChromaticMixedGraphs,Sotskov:ChromaticMixedGraphsrecent,Beck:strict-chromatic-poly-digraph}.
The $B$-polynomial of mixed graphs can be specialized to the chromatic polynomial of mixed graph (in the same way as the Tutte polynomial of graphs can be specialized to the chromatic polynomial; see Proposition~\ref{prop:chrom-partial}). In particular, by a result of Beck, Bogard, and Pham~\cite{Beck:strict-chromatic-poly-digraph}, one of the specializations of the $B$-polynomial gives the number of ways of orienting the unoriented edges of a mixed graph to get an acyclic orientation (see Corollary~\ref{cor:T20}). There is also a dual result about totally cyclic orientations of mixed graphs (see Corollary~\ref{cor:T02}), which generalizes a result of Las Vergnas~\cite{LasVergnas:Tutte(02)}.
\item The $B$-polynomial has a \emph{quasi-symmetric function} generalization $B_D(\xx;y,z)$. 
The invariant $B_D(\xx;y,z)$ generalizes to digraphs the \emph{Tutte symmetric function} defined by Stanley for graphs~\cite{Stanley:symmetric-chromatic-more}, and the \emph{chromatic quasisymmetric function} defined by Shareshian and Wachs for acyclic digraphs~\cite{Shareshian-Wachs:chromatic-quasisym-functions}. 
\end{itemize}

Besides defining and studying a natural generalization of the Tutte polynomial to digraphs, we attempt in this paper to shed a new light on some classical results about the Tutte polynomial of graphs. We are particularly motivated by certain classical results about the evaluations of the Tutte polynomial counting various classes of orientations, which we recover and extend in Section \ref{sec:Ehrhart}. Our approach in that section is reminiscent of the method developed in \cite{Beck:inside-out-polytopes} for explaining the seminal result of Stanley about the evaluations of the chromatic polynomial at negative integers \cite{Stanley:acyclic-orientations}. See for instance Remarks \ref{rk:inside-out} and \ref{rk:Potts-compatible-colorings}.

Actually, the $B$-polynomial is not the only natural generalization of the Tutte polynomial to digraphs. Indeed, in Section~\ref{sec:B-family}, we present an infinite family of digraph invariants satisfying Theorem~\ref{thm:Bpoly-Tutte}.
The $B$-polynomial is merely the ``simplest'' invariant among this family.
In the upcoming paper~\cite{Awan-Bernardi:A-poly}, we investigate another invariant from this infinite family, which we call \emph{$A$-polynomial}. Unlike the $B$-polynomial, the $A$-polynomial is an \emph{oriented matroid invariant}, that is, it only depends on the oriented matroid associated with the digraph.

The paper is organized as follows. 
In Section~\ref{sec:defs-graphs}, we set up our notation about digraphs.
In Section~\ref{sec:basics}, we define the $B$-polynomial, and explore its most immediate properties. In particular we prove the relations stated in Theorem~\ref{thm:Bpoly-Tutte} between the $B$-polynomial and the Tutte polynomial. 
We also define two invariants $T^{(1)}_D(x,y)$ and $T^{(2)}_D(x,y)$ in terms of the specialization $B_D(q,y,1)$, and we explain their relations to the Tutte polynomial of graphs and the chromatic polynomial of mixed graphs.
In Section~\ref{sec:recurrence}, we explore the consequences of recurrence relations for the $B$-polynomial. 
In Section~\ref{sec:expansions}, we give several expansions of the $B$-polynomial in terms of the oriented chromatic polynomials.
In Section~\ref{sec:Ehrhart}, we use Ehrhart theory to give an interpretation of the evaluations of the $B$-polynomial at negative values of $q$. We show that various generating functions of acyclic and totally cyclic ``modifications'' of $D$ can be obtained from the $B$-polynomial. We also prove that the $B$-polynomial of a planar digraph and its dual satisfy a partial symmetry relation. 
In Section~\ref{sec:Tutte-eval}, we interpret several evaluations of the invariants $T^{(1)}_D(x,y)$ and $T^{(2)}_D(x,y)$ in terms of orientations of mixed graphs.
In Section~\ref{sec:quasisym}, we present a quasisymmetric function generalization of the $B$-polynomial, and use the theory of $P$-partitions to prove some new properties.
In Section~\ref{sec:B-family}, we present a family of invariants generalizing the $B$-polynomial. We show that some of them are oriented matroid invariants.
We conclude in Section~\ref{sec:conclusion} with a list of open problems.

\smallskip

\section{Notation and definitions about graphs and digraphs}\label{sec:defs-graphs}
In this section we set some basic notation. 
We denote by $\NN$ the set of non-negative integers, and by $\PP$ the set of positive integers. For a positive integer $n$, we use the symbol $[n]$ to denote the set $\{1,\ldots,n\}$. For integers $a<b$, we use the symbol $[a..b]$ to denote the set $\{a,a+1,\ldots,b\}$. We denote by $\fS_n$ the set of permutations of $[n]$.

We denote by $|S|$ the cardinality of a set $S$. For sets $R,S,T$, we write $R\uplus S=T$ to indicate that $T$ is the \emph{disjoint union} of $R$ and $S$. 

For a polynomial $P$ in a variable $x$, we denote by $[x^k]P$ the coefficient of $x^k$ in $P$, and we denote by $\deg_x(P)$ its \emph{degree} in the variable $x$. 
For a condition $C$, the symbol $\ONE_C$ has value 1 if the condition $C$ is true, and 0 otherwise.

A \emph{graph} is a pair $G=(V,E)$, where $V$ is a finite set of \emph{vertices} and $E$ is a finite set of \emph{edges} which are subsets $\{u,v\}$ of 1 or 2 vertices (we accept $u=v$).
We denote by $\vv(G)$ the number of vertices of $G$, and by $\comp(G)$ the number of connected components of $G$. 

A \emph{directed graph}, or \emph{digraph} for short, is a pair $(V,A)$, where $V$ is a finite set of \emph{vertices} and $A$ is a finite set of \emph{arcs} which are pairs $(u,v)$ of vertices. We authorize our digraphs to have \emph{loops} (that is, arcs of the form $(u,u)$) and \emph{multiple arcs} (so that $A$ is really a \emph{multiset} which can contain multiple copies of each element $(u,v)$). 
We say that the arc $a=(u,v)$ has \emph{initial vertex} $u$, \emph{terminal vertex} $v$, and \emph{endpoints} $u$ and $v$. 
We say that the arc $(v,u)$ is the \emph{opposite} of the arc $(u,v)$.

The \emph{underlying graph} of a digraph $D=(V,A)$, denoted by $\uD$, is the graph which is 
obtained by replacing each arc $(u,v)$ by the edge $\{u,v\}$. 
An \emph{orientation} of a graph $G$ is a digraph $D$ with underlying graph $G$; it is obtained by choosing a direction for each edge $\{u,v\}$.
For a graph $G=(V,E)$, we denote by $\orient{G}=(V,A)$ the digraph obtained by replacing each edge $\{u,v\}\in E$ by the two opposite arcs $(u,v)$ and $(v,u)$. This operation identifies the set of graphs with the subset of digraphs such that for all $u,v\in V$ the arc $(u,v)$ appears with the same multiplicity as the opposite arc $(v,u)$.

Let $D=(V,A)$ be a digraph.
A set of arcs $S\subseteq A$ is a \emph{cocycle} if there exists a partition of the vertex set $V=V_1\uplus V_2$ such that $S$ is the set of edges with one endpoint in $V_1$ and one endpoint in $V_2$. It is a \emph{directed cocycle} if moreover the initial vertices of the arcs in $S$ are all in the same subset of vertices, say $V_1$. An arc $a\in A$ is called \emph{cyclic} if it is in a directed cycle, and \emph{acyclic} if it is in a directed cocycle. As is well known, any arc is either cyclic or acyclic, but never both. The digraph $D$ is \emph{acyclic} if every arc is acyclic (i.e. $D$ has no directed cycles), and \emph{totally cyclic} if every arc is cyclic (i.e. $D$ has no directed cocycles).

We now define the \emph{deletion}, \emph{contraction}, and \emph{reorientation} of digraphs. These operations are represented in Figure~\ref{fig:deletion-contraction}. Let $D=(V,A)$ be a digraph, and let $a=(u,v)\in A$ be an arc.
\begin{itemize}
\item We denote by $D_{\setminus a}$ the digraph obtained by \emph{deleting} $a$, that is, removing $a$ from $A$.
\item We denote by $D_{/a}$ the digraph obtained by \emph{contracting} $a$, that is, removing $a$ from $A$ and identifying its two endpoints $u$ and $v$. Note that if $u=v$, then $D_{/a}=D_{\setminus a}$.
\item We denote by $D^{-a}$ the digraph obtained by \emph{reorienting} $a$, that is, replacing $a=(u,v)$ by the opposite arc $-a=(v,u)$.
\end{itemize}
Note that these three operations commute with each other. For instance, for any distinct arcs $a,b\in A$, $D_{/a/b}=D_{/b/a}$ and $D_{\del a/b}=D_{/b\del a}$.
Hence, the definition of deletion, contraction and reorientation can be extended to sets of arcs: for any disjoint sets of arcs $R,S,T\subseteq A$ we denote by $D_{\del R/S}^{-T}$ the digraph obtained from $D$ by deleting the arcs in $R$, contracting the arcs in $S$, and reorienting the arcs in $T$.

\fig{width=.6\linewidth}{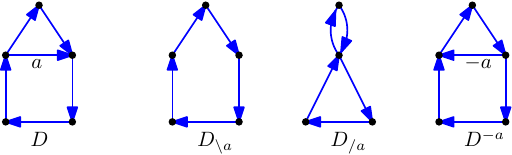}{The result of deleting, contracting, and reorienting an arc $a$ of $D$.}


\smallskip

\section{The $B$-polynomial and its relation to the Potts and Tutte polynomials}\label{sec:basics}
In this section we define the $B$-polynomial of digraphs, and study its most immediate properties. In particular, we establish the existence of a polynomial $B_D$ satisfying~\eqref{eq:def-B}, and prove Theorem~\ref{thm:Bpoly-Tutte}.

Let $D=(V,A)$ be a digraph, and let $q$ be a positive integer. 
We call a function from $V$ to $[q]$ a \emph{$q$-coloring of $D$}, and we call $f(v)$ the \emph{color} of the vertex $v$.
For any function from $V$ to $\ZZ$, we denote by $f_A^{\As}$ the set of arcs $(u,v)\in A$ such that $f(v)>f(u)$. 
We define the set of arcs $f_A^{\Ds}$, $f_A^{\wAs}$, $f_A^{\wDs}$, $f_A^=$, $f_A^{\neq}$ similarly. 
We call the elements in $f_A^{\As}$ and $f_A^{\Ds}$ the \emph{ascents} and \emph{descents} of $f$ respectively. 
A $q$-coloring $f$ is \emph{proper} if $f_A^==\emptyset$.

\begin{thm}\label{thm:existence-B}
Let $D=(V,A)$ be a digraph. There exists a (unique) trivariate polynomial $B_D(q,y,z)$ such that for all positive integers~$q$,
$$\B_D(q,y,z)=\sum_{f: V\to [q]}y^{|f_A^{\As}|}z^{|f_A^{\Ds}|}.$$
where the sum is over all $q$-colorings of $D$. We call $B_D(q,y,z)$ the \emph{$B$-polynomial} of $D$.
\end{thm}

\begin{proof}
The proof of Theorem~\ref{thm:existence-B} is based on the reduction from general $q$-colorings to \emph{surjective $q$-colorings} (surjective functions from $V$ to $[q]$). 
Let $\Surj(V,q)$ be the set of surjective $q$-colorings. 
Given an arbitrary $q$-coloring $f$, we consider the number of colors used $p=|f(V)|\leq |V|$, and the unique order-preserving bijection $\varphi$ from $f(V)$ to $[p]$. Let $\tf$ be the surjective $p$-coloring $\varphi\circ f$. Observe that $|f_A^{\As}|=|\tf_A^{\As}|$ and $|f_A^{\Ds}|=|\tf_A^{\Ds}|$. Hence,
\begin{eqnarray*}
\sum_{f: V\to [q]}y^{|f_A^{\As}|}z^{|f_A^{\Ds}|}&=&\sum_{p=1}^{|V|}~\sum_{g\in \Surj(V,p)}~\sum_{f:[V]\to [q]\textrm{ such that } \tf=g}y^{|f_A^{\As}|}z^{|f_A^{\Ds}|},\nonumber\\
&=& \sum_{p=1}^{|V|}~\sum_{g\in \Surj(V,p)}y^{|g_A^{\As}|}z^{|g_A^{\Ds}|}\left|\{f:[V]\to [q] \textrm{ such that } \tf=g\}\right|,\nonumber\\
&=& \sum_{p=1}^{|V|}{q\choose p}\sum_{g\in \Surj(V,p)}y^{|g_A^{\As}|}z^{|g_A^{\Ds}|},
\end{eqnarray*}
where the third line uses the observation that a $q$-coloring $f$ is uniquely identified by the surjective $p$-coloring $\tf$ and the set $f(V)$ (there are ${q\choose p}$ ways to choose this set). Thus, the trivariate polynomial
\begin{equation}\label{eq:q-falling-factorial}
B_D(q,y,z)=\sum_{p=1}^{|V|}\frac{q(q-1)\ldots (q-p+1)}{p!}\sum_{g\in \Surj(V,p)}y^{|g_A^{\As}|}z^{|g_A^{\Ds}|},
\end{equation}
satisfies the properties of the theorem. The uniqueness of $B_D$ is obvious since distinct polynomials in $q$ cannot agree on infinitely many values of $q$. 
\end{proof}

We now state some immediate properties of the $B$-polynomial.
\begin{prop}\label{prop:easy} 
For any digraph $D=(V,A)$, the $B$-polynomial has the following properties.
\begin{compactenum}[(a)]
\item\label{it1} $B_D(q,y,z)=B_D(q,z,y)$. 
\item\label{it2} Reorienting all the arcs of $D$ does not change the $B$-polynomial: $B_{D^{-A}}(q,y,z)=B_D(q,y,z)$.
\item\label{it3} $\ds \sum_{f: V\to [q]}x^{\l|f_A^=\r|}y^{\l|f_A^{\As}\r|}z^{\l|f_A^{\Ds}\r|}=x^{|A|}B_D\l(q,\frac{y}{x},\frac{z}{x}\r)$.
\item\label{it4} $\ds \sum_{f: V\to [q]}y^{\l|f_A^{\wAs}\r|}z^{\l|f_A^{\wDs}\r|}=(yz)^{|A|}B_D\l(q,\frac{1}{y},\frac{1}{z}\r)$.
\item\label{it5} The $B$-polynomial of the digraph with a single vertex and no arcs is $q$.
\item\label{it6} If $a=(u,u)\in A$ is a loop, then $B_{D_{\setminus a}}(q,y,z)=B_D(q,y,z)$.
\item\label{it7} If $D$ is the disjoint union of two digraphs $D_1$ and $D_2$, then $B_{D}(q,y,z)=B_{D_1}(q,y,z)\,B_{D_2}(q,y,z)$.
\item\label{it8} $\deg_q(B_D(q,y,z))=|V|$ and $B_D(q,1,1)=q^{|V|}$. If $D$ has no loops, $\deg_y(B_D(q,y,y))=|A|$.
\item\label{it9} The polynomial $B_D(q,y,z)$ is divisible by $q^{\comp(D)}$, and $B_D(q,0,0)=q^{\comp(D)}$.
\item\label{it10} The expansion of the polynomial $B_D(q,y,z)$ in the basis $\{\frac{q(q-1)\ldots (q-p+1)}{p!}y^iz^j\}_{p>0,i,j\geq 0}$ has positive integer coefficients. 
\item\label{it11} $\ds |V|![q^{|V|}]B_D(q,y,z)=\sum_{\textrm{bijection }f:V\to[|V|] }y^{|f_A^{\As}|}z^{|f_A^{\Ds}|}$.
\item\label{it12} $B_D(2,1,0)$ is the number of \emph{directed cuts} of $D$ (that is, the number of subsets of vertices $U\subset V$ such that every arc joining $U$ to $V\setminus U$ is oriented away from $U$).
\end{compactenum}
\end{prop}

\begin{proof}
Property~\eqref{it1} is easy to prove using the involution on $q$-colorings which changes the color $i\in[q]$ by $q+1-i$ (so that ascents become descents and vice-versa). The same reasoning gives~\eqref{it2}.
Property~\eqref{it3} is clear from the fact that $|f_A^=|=|A|-|f_A^{\As}|-|f_A^{\Ds}|$. 
Property~\eqref{it4} is clear from~\eqref{it1} and~\eqref{it3}. It shows that the $B$-polynomial can be equivalently thought as counting $q$-colorings according to the number of \emph{weak ascents} and \emph{weak descents}. The properties~\eqref{it5},~\eqref{it6},~\eqref{it7} are obvious from the definitions. 
For~\eqref{it8}, observe that $B_D(q,1,1)$ counts all $q$-colorings, hence $B_D(q,1,1)=q^{|V|}$. Moreover $\deg_q(B_D(q,y,z))$ cannot be more than $|V|$ by~\eqref{eq:q-falling-factorial}. Clearly $\deg_y(B_D(q,y,y))\leq |A|$, and considering proper $q$-colorings gives equality for loopless digraphs. 
For~\eqref{it9}, observe that~\eqref{eq:q-falling-factorial} shows that $q$ divides $B_D(q,y,z)$. Hence, by~\eqref{it7}, $q^{\comp(D)}$ divides $B_D(q,y,z)$. 
Moreover, $B_D(q,0,0)$ counts $q$-colorings which are constant on each connected components, and there are $q^{\comp(D)}$ such $q$-colorings.
Property~\eqref{it10} and~\eqref{it11} follow directly from~\eqref{eq:q-falling-factorial}. Lastly,~\eqref{it12} is clear from the definitions upon identifying each 2-coloring $f$ with the subset of vertices $U_f=\{v\in V~|~f(v)=1\}$. 
\end{proof}

\begin{remark}\label{rk:not-detected}
We will see that a number of properties of a digraph $D$ (acyclicity, maximal length of directed paths, etc.) can be read off the invariant $B_D(q,y,z)$. However, Property~\eqref{it2} of Proposition~\ref{prop:easy} shows that certain properties cannot be read off $B_D(q,y,z)$. In particular, the outdegree distribution of $D$, the number of sources of $D$, or the number of \emph{directed spanning trees} (spanning trees such that every vertex except one has indegree 1) cannot be read off $B_D(q,y,z)$.
\end{remark}

\begin{example}
Let us illustrate the significance of Proposition~\ref{prop:easy}\eqref{it11} for the digraphs $D$, $D'$, $D''$ represented in Figure~\ref{fig:exp-permut-stats}. Note that, up to assuming $V=[n]$ for some integer $n$, the sum in~\eqref{it11} is over the permutations of $[n]$; and these permutations are counted according to some statistics described by~$A$.
For instance, if $D$ is a directed path then the statistics are the number of ascents and descents. More precisely, if $D=([n],A)$ is the directed path with arc set $A=\{(i,i+1)~|~i\in[n-1]\}$, then 
\begin{equation}\label{eq:exp-path}
n![q^{n}]B_D(q,y,z)=\sum_{\si\in \fS_n}y^{\asc(\si)}z^{\des(\si)},
\end{equation}
where $\asc(\si)$ and $\des(\si)$ are the number of \emph{ascents} ($i\in[n-1]$ such that $\si(i)<\si(i+1)$) and \emph{descents} ($i\in[n-1]$ such that $\si(i)>\si(i+1)$) of the permutation $\si$. 

If $D'=([n],A)$ is the digraph with $i$ copies of the arc $(i,i+1)$ for all $i\in[n-1]$, then 
\begin{equation}\label{eq:exp-thickening-path}
n![q^{n}]B_{D'}(q,y,z)=\sum_{\si\in \fS_n}y^{{n \choose 2}-\maj(\si)}z^{\maj(\si)},
\end{equation}
where $\ds \maj(\si)=\sum_{\substack{i\in[n-1]\\ \si(i)>\si(i+1)}}\!\!\!\!\!\!\!i~$ is the \emph{major index}. By a classical formula (see~\cite[Prop 1.4.6]{Stanley:volume1-ed2}), the right-hand side of~\eqref{eq:exp-thickening-path} is $\prod_{k=1}^{n}\l(\sum_{i=1}^{k} y^{i-1}z^{k-i}\r)$.

Lastly, if $D''=([n],\{(u,v) ~|~ 1\leq u<v\leq n\})$, then 
\begin{equation}\label{eq:exp-Kn}
n![q^{n}]B_{D''}(q,y,z)=\sum_{\si\in \fS_n}y^{{n \choose 2}-\inv(\si)}z^{\inv(\si)},
\end{equation}
where $\inv(\si)$ is the number of \emph{inversions} (pairs $(i,j)\in[n]^2$ with $i<j$ and $\si(i)>\si(j)$). 
By a classical formula (see~\cite[Cor 1.3.13]{Stanley:volume1-ed2}), the right-hand side of~\eqref{eq:exp-Kn} is again $\prod_{k=1}^{n}\l(\sum_{i=1}^{k} y^{i-1}z^{k-i}\r)$.
These formulas will be refined in Section~\ref{sec:quasisym} (see Example~\ref{exp:exp-permut-stats-quasi}). 
\end{example}
 
\fig{width=.9\linewidth}{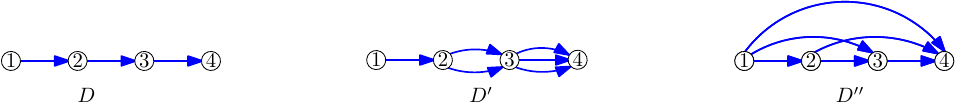}{Digraphs giving the descent, major index, and inversion statistics of permutations.}

Next, we prove Theorem~\ref{thm:Bpoly-Tutte}. Recall that the \emph{Tutte polynomial} of a graph $G=(V,E)$ is 
\begin{equation}\label{eq:def-Tutte}
T_G(x,y) = \sum_{S\subseteq E} (x-1)^{\comp(S)-\comp(E)}(y-1)^{|S|+\comp(S)-|V|},
\end{equation}
where the sum is over all subsets $S$ of edges, and $\comp(S)$ is the number of connected components of the subgraph $(V,S)$.
We refer the reader to~\cite{Welsh:PottsTutte} or~\cite[Chapter 10]{Bollobas:Tutte-poly} for an introduction to the theory of the Tutte polynomial. 
As shown by Fortuin and Kasteleyn~\cite{Fortuin:Tutte=Potts}, the Tutte polynomial is equivalent to the \emph{partition function of the Potts model on $G$}, or \emph{Potts polynomial of $G$} for short.
The \emph{Potts polynomial of $G=(V,E)$} is the unique bivariate polynomial $P_G(q,y)$ such that for all positive integers~$q$, 
\begin{equation}\label{eq:def-Potts}
P_G(q,y)= \sum_{f:V\to [q]}y^{|f_E^{\neq}|},
\end{equation}
where $f_E^{\neq}$ represents the set of edges with endpoints of different colors\footnote{In the literature the Potts polynomial of $G$ is more often defined as $ \widetilde{P}_G(q,y)=\sum_{f:V\to [q]}y^{|f_E^=|}$. This is equivalent to the convention in the present paper via the relation $P(q,y)=y^{|E|}\widetilde{P}_G(q,1/y)$.}. 
Indeed, the Tutte polynomial $T_G$ is related to $P_G$ by the following change of variables:
\begin{equation}\label{eq:Potts=Tutte}
T_G(x,y) = \frac{y^{|E|}}{(y-1)^{|V|}(x-1)^{\comp(G)}}P_G((x-1)(y-1),1/y).
\end{equation}

We now prove the first equation in Theorem~\ref{thm:Bpoly-Tutte} showing that the $B$-polynomial generalizes the Potts function of graphs (or equivalently the Tutte polynomial of graphs). 
\begin{prop}\label{PottsOne}
Let $G$ be a graph, and let $\orient{G}$ be the corresponding digraph. 
Then, 
\begin{equation*}
\B_{\orient{G}}(q,y,z) = P_G(q,yz).
\end{equation*}
\end{prop}
\begin{proof} We write $G=(V,E)$ and $\orient{G}=(V,A)$. 
We need to prove that for any positive integer~$q$,
$$\ds \sum_{f:V\to [q]}y^{\l|f_A^{\As}\r|}\,z^{\l|f_A^{\Ds}\r|}=\sum_{f:V\to [q]}(yz)^{\l|f_E^{\neq}\r|}.$$
where $f_E^{\neq}=\{\{u,v\}\in E~|~ f(u)\neq f(v)\}$. 
Hence, it suffices to prove that for any $q$-coloring $f$, 
\begin{equation}\label{eq:proof-B=Potts}
y^{\l|f_A^{\As}\r|}\,z^{\l|f_A^{\Ds}\r|}=(yz)^{\l|f_E^{\neq}\r|}.
\end{equation}
Consider an edge $e=\{u,v\}$ of $G$ and the corresponding arcs $a=(u,v)$ and $b=(v,u)$ of $D$. If $u$ and $v$ have distinct colors, then one of the arcs $a,b$ will contribute a factor of $y$ to the left-hand side of~\eqref{eq:proof-B=Potts}, whereas the other will contribute $z$. Thus the arcs $a,b$ contribute a factor $yz$, as does~$e$. On the other hand, if $u$ and $v$ have the same color, then the arcs $a,b$ contribute $1$, as does~$e$. 
\end{proof}

We now prove the second relation of Theorem~\ref{thm:Bpoly-Tutte} showing that the Potts polynomial of a graph $G$ is equivalent to the average of the $B$-polynomial over the orientations of $G$.
\begin{prop}\label{PottsTwo}
Let $G=(V,E)$ be a graph. Then,
\begin{equation*}
\frac{1}{2^{|E|}}\sum_{\vec{G}\in\ori(G)} \B_{\vec{G}}(q,y,z)
= \Potts_G\left(q,\frac{y+z}{2}\right).
\end{equation*}
\end{prop}
\begin{proof}
For any positive integer $q$ we can write
$$\sum_{\vec{G}\in \ori(G)} B_D(q,y,z)=\sum_{f:V\to[q]}~\sum_{\vec{G}=(V,A)\in \ori(G)}y^{|f_A^{\As}|}\,z^{|f_A^{\Ds}|}.$$
Moreover, for any $q$-coloring $f$, 
$$\sum_{\vec{G}=(V,A)\in \ori(G)}y^{|f_A^{\As}|}\,z^{|f_A^{\Ds}|}=\sum_{\vec{G}=(V,A)\in \ori(G)}~\prod_{(u,v)\in A}w_f(u,v)=\prod_{\{u,v\}\in E}\l(w_f(u,v)+w_f(v,u)\r),$$
where $w_f(u,v)=y$ if $f(v)>f(u)$, $w_f(u,v)=z$ if $f(v)<f(u)$, and $w_f(u,v)=1$ if $f(v)=f(u)$.
Consider an edge $e=\{u,v\}$ of $G$ and the corresponding arcs $a=(u,v)$ and $b=(v,u)$ of $D$. If $u$ and $v$ have different colors then $w_f(u,v)+w_f(v,u)=(y+z)$, whereas if $u$ and $v$ have the same color then $w_f(u,v)+w_f(v,u)=2$. Thus,
$$
\sum_{\vec{G}} B_D(q,y,z)=\sum_{f:V\to[q]}(y+z)^{\#\{u,v\}\in E,\, f(v)\neq f(u)}\,2^{\#\{u,v\}\in E,\, f(v)= f(u)}=2^{|E|}\Potts_G\left(q,\frac{y+z}{2}\right),
$$
as desired.
\end{proof}

Next we prove the third equation in Theorem~\ref{thm:Bpoly-Tutte}. 
\begin{prop}\label{PottsThree}
Let $D=(V,A)$ be a digraph, and let $\uD=(V,E)$ be the underlying graph. Then, $$\B_D(q,y,y) = \Potts_\uD(q,y).$$
\end{prop}

\begin{proof}
For any positive integer $q$,
$$B_D(q,y,y)=\sum_{f:V\to[q]} y^{|f_A^{\As}|+|f_A^{\Ds}|}=\sum_{f:V\to[q]} y^{|f_A^{\neq}|}=P_\uD(q,y),$$
as desired.
\end{proof}

\begin{remark} \label{rk:PottsThreedetects} 
Proposition~\ref{PottsThree} implies that the polynomial $B_D(q,y,z)$ contains all the information given by the Tutte polynomial of the underlying graph $\uD$ (assuming one knows $|A|$).
For instance, the number of spanning trees of $\uD$ can be read off $B_D(q,y,z)$. However, it is clear that $B_D(q,y,z)$ contains more information than $T_\uD(x,y)$, as Proposition~\ref{prop:easy}\eqref{it12} already exemplifies.
\end{remark}

Note that the third variable $z$ is ``redundant'' in the relations between $B$-polynomial and the Potts polynomial given in~\eqref{eq:PottsOne}, and~\eqref{eq:PottsTwo}. Indeed, for any graph $G=(V,E)$ we get 
\begin{equation}\label{eq:By-Potts}
P_G(q,y)=\B_{\orient{G}}(q,y,1)=\frac{1}{2^{|E|}}\sum_{\vec{G}\in\ori(G)}\B_{\vec{G}}(q,2y-1,1).
\end{equation}

We now introduce two variants of $B_D(q,y,1)$ which will be convenient in order to state relations between the $B$-polynomial and the Tutte polynomial. For this, we adopt a different perspective on digraphs, by seeing them as \emph{mixed graphs} (a.k.a. \emph{partially oriented graphs}). A \emph{mixed graph} is a graph with oriented edges and unoriented edges. Equivalently, \emph{mixed graphs} are digraphs $D=(V,A)$ together with an \emph{arc-partition} $E(D)$ of the arc-set $A$ into singletons called \emph{oriented edges}, and doubletons made of two opposite arcs called \emph{unoriented edges}. In particular, any graph $G$ identifies with the mixed graph $\orient{G}$ without oriented edges. A \emph{complete orientation} of a mixed graph $D$, is a digraph obtained by choosing a direction for each unoriented edge (i.e. replacing each doubleton in $E(D)$ by a singleton). 
We denote by $\ori(D)$ the set of complete orientations of~$D$. 

\begin{defn}
Let $D=(V,A)$ be a mixed graph with arc-partition $E(D)$. We define 
\begin{equation}\label{eq:defT1}
T_D^{(1)}(x,y) \defeq \frac{y^{|E(D)|}}{(y-1)^{|V|}(x-1)^{\comp(D)}} B_D\left((x-1)(y-1),\frac{1}{y},1\right),
\end{equation}
and
\begin{equation}\label{eq:defT2}
T_D^{(2)}(x,y) \defeq \frac{y^{|E(D)|}}{2^{|E(D)|}(y-1)^{|V|}(x-1)^{\comp(D)}}\sum_{\vec{D} \in \ori(D)} B_{\vec D}\left((x-1)(y-1),\frac{2-y}{y},1\right).
\end{equation}
\end{defn}

\begin{example}\label{exp:exp-intro-Tutte}
For the digraph $D$ with a single arc represented in Figure~\ref{fig:exp-intro} (thought as a mixed graph with no unoriented edges), one finds $T^{(1)}_D(x,y)=(xy+x-y)/2$ and $T^{(2)}_D(x,y)=x/2$. For the digraph $D''$ represented in Figure~\ref{fig:exp-intro} (thought as a mixed graph with one unoriented edges), one finds 
\begin{eqnarray}
T^{(1)}_{D''}(x,y)&=&\frac{{x}^{2}{y}^{3}+2\,{x}^{2}{y}^{2}-2\,x{y}^{3}+2\,{x}^{2}y-x{y}^{2}+{y}^{3}+{x}^{2}-xy-{y}^{2}+x+y}{6},\label{eq:exp-intro-T1}\\
T^{(2)}_{D''}(x,y)&=&\frac{-{x}^{2}{y}^{2}+2\,{x}^{2}y-x{y}^{2}+2\,{x}^{2}+2\,xy+2\,{y}^{2}+2\,x+2\,y}{12}.\label{eq:exp-intro-T2}
\end{eqnarray}
\end{example}

Relations~\eqref{eq:PottsOne} and~\eqref{eq:PottsTwo} can be rewritten as follows via~\eqref{eq:Potts=Tutte}: 
\begin{cor}
For any unoriented graph $G$, $T_{\orient{G}}^{(1)}(x,y) = T_{\orient{G}}^{(2)}(x,y) = T_G(x,y)$.
\end{cor}



It is clear from the definition that the invariants $T_D^{(1)}$ and $T_D^{(2)}$ are rational functions, and we will prove later that they are actually \emph{polynomials} in $x$ and $y$ (see Proposition~\ref{prop:polynomiality}). 
We now show that these polynomials coincide for $y=0$ and count \emph{strictly compatible colorings}. Given a mixed graph $D=(V,A)$, we call a $q$-coloring $f$ of $D$ \emph{strictly compatible} if for every unoriented edge the endpoints have different colors, and for every oriented edge the color of the initial vertex is less than the color of the terminal vertex. We denote by $\partialchrom_D(q)$ the number of strictly compatible $q$-colorings of the mixed graph $D$. This invariant was already studied in~\cite{Sotskov:ChromaticMixedGraphs,Sotskov:ChromaticMixedGraphsrecent,Beck:strict-chromatic-poly-digraph}.
Note that for an unoriented graph $G$, the strictly compatible $q$-colorings are the \emph{proper} $q$-colorings, so that $\partialchrom_G(q)$ is the chromatic polynomial, which is known to be related to the Tutte polynomial by $$\ds \partialchrom_G(q)=(-1)^{|V|-\comp(D)}q^{\comp(D)}T_G(1-q,0).$$ 
The following is a generalization of this identity to mixed graphs. 
\begin{prop}\label{prop:chrom-partial}
For any mixed graph $D$,
$$\partialchrom_D(q)~=~(-1)^{|V|-\comp(D)}q^{\comp(D)}T^{(1)}_D(1-q,0)~=~(-1)^{|V|-\comp(D)}q^{\comp(D)}T^{(2)}_D(1-q,0).$$ 
In particular, the polynomials $T_D^{(1)}(x,0)$ and $T_D^{(2)}(x,0)$ are equal, and $\partialchrom_D(q)$ is a polynomial.
\end{prop}

\begin{proof}
It suffices to prove this relation for every  positive integer $q$.
Let $E(D)$ be the arc-partition of $D$ and let $e=|E(D)|$. 
Observe that for any $q$-coloring $f$, one has $|f_A^{\As}|\leq e$, with equality if and only if $f$ is strictly compatible. This is because in any pair of opposite arcs $\{a,-a\}$ in $E(D)$, at most one arc can be in $f_A^{\As}$. 
Thus, 
\begin{eqnarray*}
\partialchrom_D(q)&=&[y^{e}]B_D(q,y,1)=[y^{0}]y^{e}B_D(q,1/y,1)=[y^{0}](y-1)^{|V|-\comp(D)}q^{\comp(D)}T_D^{(1)}\l(\frac{q}{y-1}+1,y\r),\\ 
&=& (-1)^{|V|-\comp(D)}q^{\comp(D)}T_D^{(1)}(1-q,0).
\end{eqnarray*}
Next, we observe that any strictly compatible coloring of $D$ is a strictly compatible coloring of a unique complete orientation $\vec{D}$ of $D$.
Thus, 
$$\partialchrom_D(q)=\sum_{\vec{D}\in\ori(D)}\!\!\! [y^{e}]B_{\vec D}(q,y,1)=[y^{0}]\frac{y^{e}}{2^e}\sum_{\vec{D}\in\ori(D)}\!\!\! B_{\vec D}\l(q,\frac{2-y}{y},1\r)=(-1)^{|V|-\comp(D)}q^{\comp(D)}T_D^{(2)}(1-q,0).$$
\end{proof}

\begin{remark}
We will see that the invariants $T^{(1)}_D$ and $T^{(2)}_D$ share many features with the Tutte polynomial of graphs. 
 Alas, unlike the Tutte polynomial of graphs, the invariants $T^{(1)}_D$ and $T^{(2)}_D$ do not have integer nor positive coefficients in general as can be seen from the above examples.
\end{remark}

\smallskip



\section{Recurrence and the Tutte activities of mixed graphs}\label{sec:recurrence}
In this section, we show that the $B$-polynomial admits a Tutte-like recurrence with respect to \emph{unoriented edges} (that is, pairs of opposite arcs). 
This leads to some notions of partial \emph{Tutte-activities}. However, no proper recurrence relation holds with respect to single arcs, so the recurrence perspective plays a much lesser role for the $B$-polynomial than for the Tutte polynomial.

We first extend the known recurrence relation for the Potts function from the setting of graphs to the general setting of digraphs.
\begin{lemma} \label{lem:rec-edge}
Let $D=(V,A)$ be a digraph. If two opposite arcs $a=(u,v)$ and $-a=(v,u)$ both belong to $A$, then 
\begin{equation}\label{eq:rec-edge} 
\B_D(q,y,z) = (yz)\B_{D_{\setminus e}}(q,y,z) + (1-yz)\B_{D_{/e}}(q,y,z),
\end{equation}
where $e=\{a,-a\}$. Note that in the special case where $u=v$, this gives $\B_D(q,y,z) = \B_{D/e}(q,y,z)$.
\end{lemma}

\begin{proof}
Let $q$ be a positive integer. For any $q$-coloring $f$ of $D$, 
$$y^{|f_A^{\As}|}z^{|f_A^{\Ds}|}=\l(yz+(1-yz)\ONE_{f(u)=f(v)}\r) \l(y^{|f_{A\setminus e}^{\As}|}z^{|f_{A\setminus e}^{\Ds}|}\r).$$
Summing over all q-colorings gives~\eqref{eq:rec-edge}. 
\end{proof}


Next, we establish a relation between $B_D$, $B_{D_{\setminus a}}$, $B_{D_{/a}}$, and $B_{D^{-a}}$.
\begin{lemma} \label{lem:rec-arc}
Let $D=(V,A)$ be a digraph. For any arc $a=(u,v)$ in $A$, 
\begin{equation}\label{eq:rec-arc} 
\B_D(q,y,z)+\B_{D^{-a}}(q,y,z) = (y+z)\B_{D\setminus a}(q,y,z) + (2-y+z)\B_{D/a}(q,y,z).
\end{equation}
Note that in the special case where $u=v$, this gives $2\B_D(q,y,z) = 2\B_{D/a}(q,y,z)$.
\end{lemma}

\begin{proof}
Let $q$ be a positive integer. For any $q$-coloring $f$ of $D$, 
$$y^{\l|f_A^{\As}\r|}z^{\l|f_A^{\Ds}\r|}+y^{\l|f_{A\cup \{-a\} \setminus a}^{\As}\r|}z^{\l|f_{A\cup \{-a\} \setminus a}^{\Ds}\r|}=(y+z+(2-y-z)\ONE_{f(u)=f(v)}) \l(y^{\l|f_{A\setminus a}^{\As}\r|}z^{\l|f_{A\setminus a}^{\Ds}\r|}\r).$$
Summing over all q-colorings gives~\eqref{eq:rec-arc}. 
\end{proof}

Note that, unlike Equation~\eqref{eq:rec-edge}, Equation~\eqref{eq:rec-arc} does not express the $B$-polynomial of $D$ in terms of the $B$-polynomial of digraphs with fewer edges. In fact, the $B$-polynomial is not the only polynomial to satisfy relation~\eqref{eq:rec-arc} (see for instance Section~\ref{sec:B-family}), and there is no proper notion of ``universality'' with respect to this type of recurrence relation. However, as we now explain, it is still possible to use~\eqref{eq:rec-edge} and~\eqref{eq:rec-arc} to define partial notions of \emph{Tutte activities} by focusing on unoriented edges.

Let us first rewrite Lemmas~\ref{lem:rec-edge} and~\ref{lem:rec-arc} for the invariant $T^{(1)}_D$ and $T^{(2)}_D$ of mixed graphs. We call \emph{bridge} of a mixed graph, an edge whose deletion increases the number of connected components.
\begin{cor}\label{cor:rec-Tutte-i}
Let $D$ be a mixed graph, and let $e=\{a,-a\}\in E(D)$ be an unoriented edge. For $i\in\{1,2\}$,
\begin{compactitem}
\item if $e$ is neither a bridge nor a loop, then 
$\ds T_D^{(i)}(x,y)=T_{D_{\setminus e}}^{(i)}(x,y)+T_{D_{/e}}^{(i)}(x,y).$
\item if $e$ is a bridge, then 
$\ds T_D^{(i)}(x,y)=(x-1)T_{D_{\setminus e}}^{(i)}(x,y)+T_{D_{/e}}^{(i)}(x,y).$
\item if $e$ is a loop, then 
$\ds T_D^{(i)}(x,y)=yT_{D_{\setminus e}}^{(i)}(x,y).$
\end{compactitem}
\end{cor}

The proof of Corollary~\ref{cor:rec-Tutte-i} for $i=1$ (resp. $i=2$) is immediate from Lemmas~\ref{lem:rec-edge} (resp. Lemma~\ref{lem:rec-arc}).
We now give expressions for $T^{(1)}_D$ and $T^{(2)}_D$ obtained by iterating Corollary~\ref{cor:rec-Tutte-i} on the set of unoriented edges.

\begin{prop}\label{prop:Tutte-subgraphs}
Let $D=(V,A)$ be a mixed graph, and let $H$ be the set of unoriented edges. Then for $i\in\{1,2\}$,
$$T^{(i)}_D(x,y) = \sum_{R\uplus S=H}(x-1)^{\comp(D_{\del R})-\comp(D)}(y-1)^{|S|+\comp(V,S)-|V|}T_{D_{\del R /S}}^{(i)}(x,y),$$
where the sum is over all partitions of $H$ into two disjoint sets $R,S$ (there are $2^{|H|}$ summands), and $\comp(V,S)$ is the number of connected components of the graph $(V,S)$.
\end{prop}

\begin{proof}
We apply the recurrence relation of of Corollary~\ref{cor:rec-Tutte-i} successively on every edge of $H$ in an arbitrary order, with the following twist: when applying the recurrence on a loop $e$ we use the identity 
\begin{equation}\label{eq:rec-loop-splited}
T_D^{(i)}(x,y)=T_{D_{\setminus e}}^{(i)}(x,y)+(y-1)T_{D_{/e}}^{(i)}(x,y).
\end{equation}
In Figure~\ref{fig:comput-tree}, we represent the process of applying this recurrence successively on every edge of $H$ by a \emph{computation tree}. The root of the computation tree is $D$, and the leaves of the computation tree correspond to all the digraphs $T_{D_{\del R /S}}^{(i)}(x,y)$ with $R\uplus S=H$. The recurrence gives 
$$T^{(i)}_D(x,y)=\sum_{R\uplus S=H}(x-1)^{\al(R,S)}(y-1)^{\be(R,S)}T_{D_{\del R /S}}^{(i)}(x,y),$$
where $\al(R,S)$ is the number of bridges deleted during the deletion-contraction process leading from $D$ to $D_{\del R /S}$, and $\be(R,S)$ is the number of loops contracted during this process. Moreover $\ds \al(R,S)=\comp(D_{\del R})-\comp(D)$ since this number starts at 0 and increases by one exactly when deleting a bridge during the deletion-contraction process (where $R$ represent the set of edges which have been deleted). Similarly $\be(R,S)=|S|+\comp(V,S)-|V|$ because this quantity starts at 0 and increases by one exactly when contracting a loop during the deletion-contraction process (where $S$ represent the set of edges which have been contracted).
\end{proof}

\fig{width=.7\linewidth}{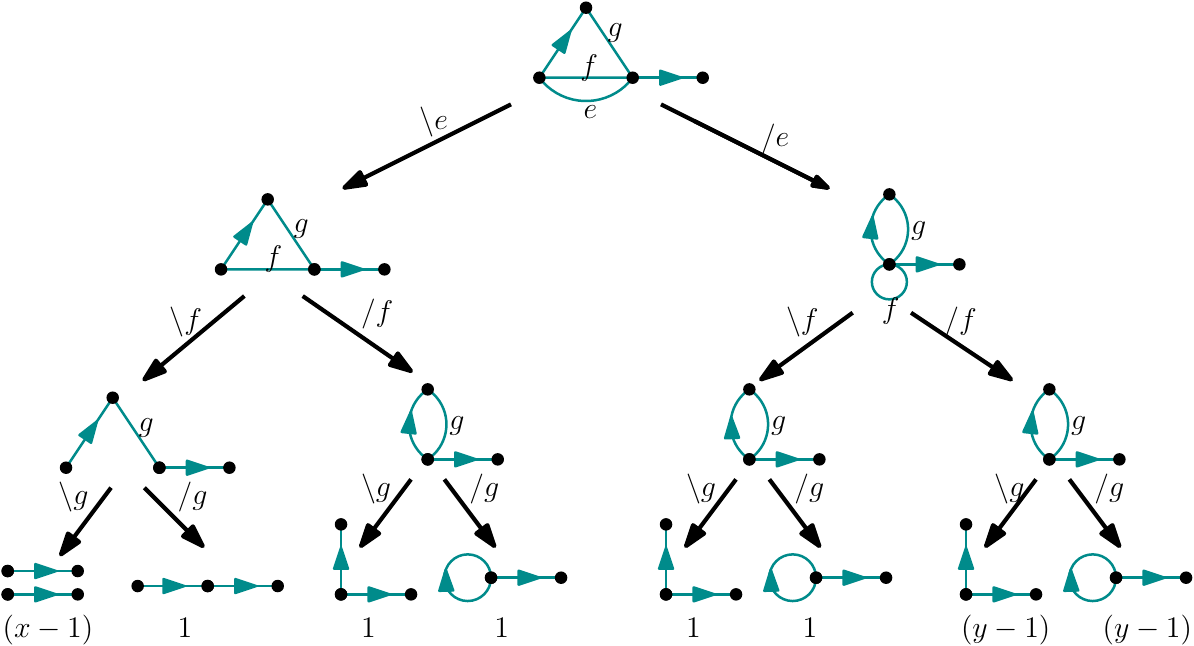}{The computation tree illustrating the computation of $T^{(i)}_D$ using the recurrence relation~\eqref{eq:rec-loop-splited} for loops. In this figure, the complete edges $e,f,g$ are represented by one segment rather than by two opposite arcs.
The factor $(x-1)^{\comp(D_{\del R})-\comp(D)}(y-1)^{|S|+\comp(V,S)-|V|}$ is indicated under each of the digraphs $D_{\del R /S}$ represented at the bottom of the computation tree.}

\begin{prop}\label{prop:Tutte-forests}
Let $D=(V,A)$ be a mixed graph, let $H$ be the set of unoriented edges, and let $\prec$ be a total order on the set $H$. 
Then for $i$ in $\{1,2\}$,
$$T^{(i)}_D(x,y) = \sum_{F\subseteq H \textrm{ forest}} (x-1)^{\comp(S)-\comp(D)}y^{\ext_\prec(F)}T^{(i)}_{D_{\setminus \overline{F}/F}}(x,y),$$
where the sum is over all \emph{forests} (that is, subset $F\subseteq H$ such that the graph $(V,F)$ has no cycles), and for a forest $F$ we write $\overline{F}:=H\setminus F$, and we denote by $\ext_\prec(F)$ the number of edges $e\in\overline{F}$ such that there is a path $P$ in $F$ between the endpoints of $e$ and $e$ is smaller than any edge in $P$ for the order $\prec$. 
\end{prop}

\fig{width=.55\linewidth}{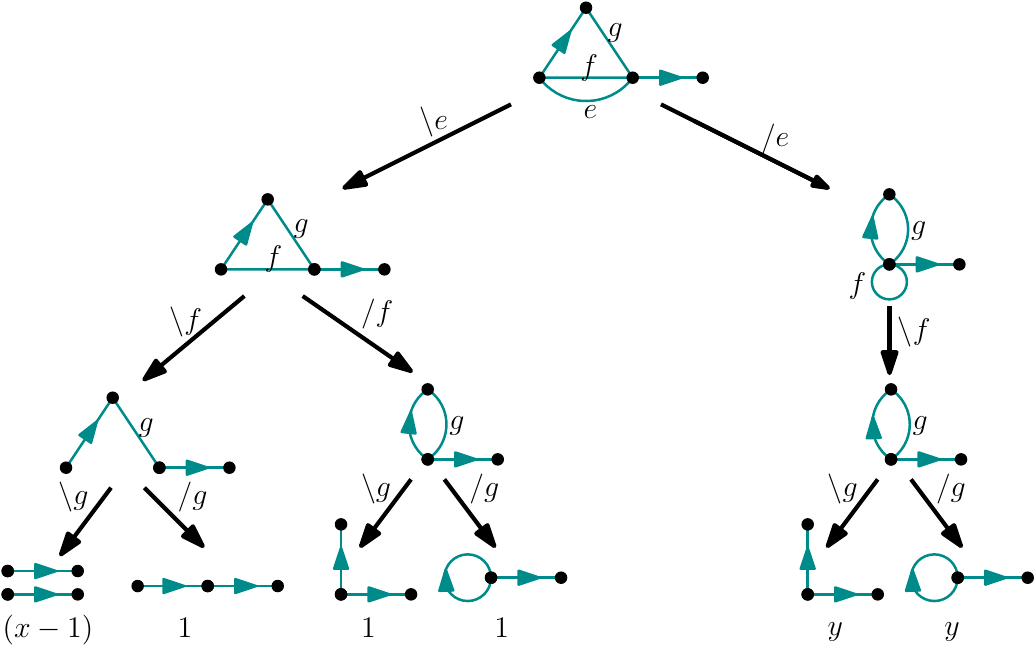}{The computation tree illustrating the computation of $T^{(i)}$ using the recurrence relation of Corollary~\ref{cor:rec-Tutte-i} for loops. The order on $H=\{e,f,g\}$ is $e\succ f \succ g$. The digraphs at the bottom of the computation tree are of the form $D_{D_{\setminus \overline{F}/F} }$ with $F\subseteq H$ such that the graph $(V,F)$ has no cycles, and the corresponding factor $(x-1)^{\comp(D_{\del \overline{F}})-\comp(D)}y^{\ext_\prec(F)}$ is indicated under each digraph.
}

\begin{proof}
We apply the recurrence relation of Corollary~\ref{cor:rec-Tutte-i} successively on every edge of $H$ in the decreasing order given by $\prec$ (that is, we start with the largest element). In Figure~\ref{fig:comput-tree2}, we represent the process of applying this recurrence successively on every edge of $H$ by a \emph{computation tree}. 
Since we use the relation $T_D^{(i)}(x,y)=yT_{D_{\setminus e}}^{(i)}(x,y)$ for loops, we are never contracting loops during this process. Hence the set $S\subseteq H$ of edges contracted during this process must be a forest (and this is the only condition on $S$).
Therefore, 
$$T^{(i)}_D(x,y)=\sum_{R\uplus S=H,~S \textrm{ forest}}(x-1)^{\al(R,S)}y^{\ga(R,S)}T_{D_{\del R /S}}^{(i)}(x,y),$$
where $\al(R,S)$ is the number of bridges deleted during the deletion-contraction process leading from $D$ to $D_{\del R /S}$, and $\ga(R,S)$ is the number of loops deleted during this process. 
Moreover, $\ds \al(R,S)=\comp(D_{\del R})-\comp(D)$ since this number starts at 0 and increases by one exactly when deleting a bridge during the deletion-contraction. Similarly, $\ga(S)=\ext_\prec(S)$ because an edge $e$ is deleted as a loop if and only if there is a path of edges in $H$ joining the endpoints of $e$ which is contracted in the deletion-contraction process before the deletion of $e$.
\end{proof}

\begin{remark}
In the case where the mixed graph $D$ has no oriented edges (i.e. $D=\orient{G}$ for a graph $G$), Proposition~\ref{prop:Tutte-subgraphs} becomes the \emph{subgraph expansion}~\eqref{eq:def-Tutte} of the Tutte polynomial:
$$T^{(i)}_D(x,y)=\sum_{S\subseteq E}(x-1)^{\comp(S)-\comp(D)}(y-1)^{|S|+\comp(S)-|V|},$$ 
while Proposition~\ref{prop:Tutte-forests} gives the \emph{forest expansion} of the Tutte polynomial:
$$T^{(i)}_D(x,y)=\sum_{F\subseteq E,\textrm{ forest}}(x-1)^{\comp(S)-\comp(D)}y^{\ext_\prec(F)}.$$
The forest expansion is given with external activities corresponding to a total order of the arcs of $D$. As for the Tutte polynomial of graphs, it would be possible to give a more general definition of ``external activity with respect to a computation tree'' in the sense of Gordon and McMahon~\cite{Gordon:intervals-Tutte-poly}.
It would also be possible to interpolate between Proposition~\ref{prop:Tutte-subgraphs} and Proposition~\ref{prop:Tutte-forests} by mimicking the \emph{generalized activity} construction of Gordon and Traldi~\cite{Gordon:generalized-activities-Tutte-poly}.
\end{remark}

\smallskip


\section{Oriented chromatic polynomials, and their generating functions}\label{sec:expansions}
In this section, we define \emph{strict} and \emph{weak} \emph{chromatic polynomials} for digraphs. We then express the $B$-polynomial in terms of these chromatic polynomials. These expressions will be used repeatedly in the subsequent sections. 

Recall that the \emph{chromatic polynomial} of a graph $G=(V,E)$, denoted by $\chi_G(q)$, is the unique polynomial whose evaluation at any positive integer~$q$ gives the number of proper $q$-colorings of $G$:
$$\chi_G(q) = \l|\{f:V\to [q]~|~ \forall \{u,v\}\in E, f(u)\neq f(v) \}\r|.$$ 
Note that $\chi_G(q) =[y^{|E|}]P_G(q,y)$. We now define some analogous polynomials for digraphs.

\begin{defn}
Let $D=(V,A)$ be a digraph. The \emph{strict-chromatic polynomial} of $D$, denoted by $\chi^{\As}_D(q)$, is the unique polynomial whose evaluation at any positive integer~$q$ gives the number of $q$-colorings of $D$ with only strict ascents:
\begin{equation*}
\chrom_D^{\As}(q) =[y^{|A|}]B_D(q,y,1)= \l|\{f:V\to [q]~|~ \forall (u,v)\in A,~ f(u)<f(v)\}\r|.
\end{equation*}
The \emph{weak-chromatic polynomial} of $D$, denoted by $\chi^{\wAs}_D(q)$, is the unique polynomial whose evaluation at any positive integer~$q$ gives the number of $q$-colorings of $D$ with only weak ascents:
\begin{equation*}
\chrom_D^{\wAs}(q) =B_D(q,0,1)= \l|\{f:V\to [q]~|~ \forall (u,v)\in A,~ f(u)\leq f(v)\}\r|.
\end{equation*}
\end{defn}

\begin{remark}
The weak and strict-chromatic polynomials are very closely related to the \emph{order polynomials} of posets, as defined by Stanley in~\cite[chap. IV]{Stanley:thesis} (see also~\cite{Stanley:chromatic-ordered}). 
Indeed, if $D$ is acyclic, then we can consider the partial ordering of vertices $\preceq_D$, where $u\preceq_D v$ means that there exists a directed path from $u$ to $v$. In this case, our strict and weak-chromatic polynomials $\chi^{\As}_G(q)$, $\chi^{\wAs}_G(q)$ coincide with the weak and strict-order polynomials associated to the poset $(V,\preceq_D)$.
\end{remark}

\begin{remark} \label{rk:inside-out}
Observe that for any graph $G$, 
\begin{equation}\label{eq:inside-out-chromatic}
\chrom_G(q)=\sum_{\vec{G}\in \ori(G)} \chrom^{\As}_{\vec{G}}(q).
\end{equation}
Indeed, for any proper $q$-coloring $f$ of $G$ there is a unique orientation of $G$ such that $f$ has only strict ascents. 
Equation~\eqref{eq:inside-out-chromatic} is one of the keystones in the \emph{inside-out} approach to the chromatic polynomial developed by Beck and Zaslavsky in~\cite{Beck:inside-out-polytopes}.
\end{remark}


\begin{remark}
Because the strict and weak-chromatic polynomials are specializations of the $B$-polynomial, they satisfy the same type of recurrence relations. 
For example, extracting the coefficient of $y^{|A|}$ in~\eqref{eq:rec-arc} gives 
$$\chrom_D^{\As}(q)+\chrom_{D^{-a}}^{\As}(q)=\chrom_{D_{\setminus a}}^{\As}(q) - \chrom_{D_{/a}}^{\As}(q).$$
\end{remark}

We now state the main result of this section.
\begin{thm}\label{thm:expansions}
Let $D=(V,A)$ be a digraph. Then,
\begin{eqnarray}
\sum_{R\uplus S\uplus T= A} y^{|S|}z^{|T|}\chrom_{D^{-T}_{\setminus R}}^{\As}(q)&=& \B_D(q,1+y,1+z),\label{>Delete}\\
\sum_{R\uplus S\uplus T= A} y^{|S|}z^{|T|}\chrom_{D^{-T}_{\setminus R}}^{\wAs}(q)&=&(1+y+z)^{|A|}\B_D\left(q,\frac{1+y}{1+y+z},\frac{1+z}{1+y+z}\right),\label{geqDelete}\\
\sum_{R\uplus S\uplus T = A} y^{|S|}z^{|T|}\chrom_{D^{-T}_{/R}}^{\As}(q) &=&\B_D(q,y,z),\label{>Contract}\\
\sum_{R\uplus S\uplus T = A} y^{|S|}z^{|T|}\chrom_{D^{-T}_{/R}}^{\wAs}(q)&=&(1+y+z)^{|A|}\B_D\left(q,\frac{y}{1+y+z},\frac{z}{1+y+z}\right),\label{geqContract}
\end{eqnarray}
where the sum is over all ways of partitioning the arc set $A$ in three subsets.
\end{thm}

\begin{proof}
For any positive integer $q$, we have 
\begin{equation}
\B_D(q,y,z) = \sum_{f:V\to [q]} \prod_{(u,v)\in A} \theta(f(u)-f(v)).
\end{equation}
where 
$$\ds
\theta(n)=
\begin{cases}
y&\textrm{if } n<0,\\
z& \textrm{if } n>0,\\
1&\textrm{if } n=0.
\end{cases}
$$
We now consider several ways of expressing $\theta$ as a sum of three terms:
\begin{eqnarray}
\theta(n) &=& x_1+x_2 \ONE_{n<0}+x_3\ONE_{n>0}, \textrm{ for } (x_1,x_2,x_3) = (1,y-1,z-1),\label{eq:<delete}\\
\theta(n) &=& x_1+x_2 \ONE_{n\leq 0}+x_3\ONE_{n\geq 0}, \textrm{ for } (x_1,x_2,x_3) = (y+z-1,1-z,1-y),\label{eq:leq-delete}\\
\theta(n) &=& x_1\ONE_{n=0}+x_2 \ONE_{n<0}+x_3\ONE_{n>0}, \textrm{ for } (x_1,x_2,x_3) = (1,y,z),\label{eq:<contract}\\
\theta(n) &=& x_1\ONE_{n=0}+x_2 \ONE_{n\leq 0}+x_3\ONE_{n\geq 0}, \textrm{ for } (x_1,x_2,x_3) = (1-y-z,y,z).\label{eq:leq-contract}
\end{eqnarray}
Each of the expressions (\ref{eq:<delete}-\ref{eq:leq-contract}) gives an equation for $B_D(q,y,z)$ which, after a suitable change of variables, give the equations (\ref{>Delete}-\ref{geqContract}) respectively. For instance, using~\eqref{eq:<delete} gives
\begin{eqnarray*}
\B(q,y,z)&=&\sum_{f:V\to [q]}\prod_{(u,v)\in A}\left(1+(y-1)\ONE_{f(u)<f(u)}+(z-1)\ONE_{f(u)>f(v)}\right),\\
&=&\sum_{f:V\to [q]}\sum_{R\uplus S\uplus T=A}(y-1)^{|S|}(z-1)^{|T|}\ONE_{S\subseteq f_A^{\As},\textrm{ and }T\subseteq f_A^{\Ds}},\\
&=&\sum_{R\uplus S\uplus T=A} (y-1)^{|S|}(z-1)^{|T|}\sum_{f:V\to [q]}\ONE_{S\subseteq f_A^{\As},\textrm{ and }T\subseteq f_A^{\Ds}},\\
&=&\sum_{R\uplus S\uplus T=A} (y-1)^{|S|}(z-1)^{|T|}\chrom_{D^{-T}_{\setminus R}}^{\As}(q).
\end{eqnarray*}
This gives~\eqref{>Delete} by a change of variables. Similarly, using~\eqref{eq:<contract} gives
\begin{eqnarray*}
\B(q,y,z)
&=&\sum_{R\uplus S\uplus T=A} y^{|S|}z^{|T|}\sum_{f:V\to [q]}\ONE_{R\subseteq f_A^=,~S\subseteq f_A^{\As},\textrm{ and }T\subseteq f_A^{\Ds}},\\
&=&\sum_{R\uplus S\uplus T=A} y^{|S|}z^{|T|}\chrom_{D^{-T}_{/R}}^{\As}(q),
\end{eqnarray*}
where in the second line we identify the $q$-colorings of $D$ such that $R\subseteq f_A^=$ with the $q$-colorings of $D_{/R}$.
This gives~\eqref{>Contract} by a change of variables. Equations~\eqref{geqDelete} and~\eqref{geqContract} are obtained similarly.
\end{proof}

Each generating function in Theorem~\ref{thm:expansions} involves two operations among deletion, contraction, and reorientation. As specializations, we can get generating functions involving a single operation. For deletion we get,
\begin{eqnarray}
\sum_{S\subseteq A} y^{|A\setminus S|} \chrom_{D_{\del S}}^{\As}(q)&=&\B_D(q,y+1,1),\label{eq:deletion-expand}\\
\sum_{S\subseteq A} y^{|A\setminus S|} \chrom_{D_{\del S}}^{\wAs} (q)&=&(y+1)^{|A|} \B_D\left(q,\frac{1}{y+1},1\right).\label{eq:deletion-expand2}
\end{eqnarray}
For contraction we get,
\begin{eqnarray}
\sum_{S\subseteq A} y^{|A\setminus S|} \chrom_{D_{/ S}}^{\As}(q)&=&\B_D(q,y,0), \label{eq:expand-contraction}\\
\sum_{S\subseteq A} y^{|A\setminus S|} \chrom_{D_{/S}}^{\wAs} (q)&=&(y+1)^{|A|}\B_D\l(q,\frac{y}{y+1},0\r).\nonumber
\end{eqnarray}
For reorientation we get
\begin{eqnarray*}
\sum_{S\subseteq A} x^{|S|} \chrom_{D^{-S}}^{\As}(q)&=&[y^{|A|}]\B_D(q,y,x y),\\
\sum_{S\subseteq A} x^{|S|} \chrom_{D^{-S}}^{\wAs}(q) &=&(1+x)^{|A|} \B_D\left(q,\frac{1}{1+x},\frac{x}{1+x}\right).
\end{eqnarray*}

Next, we state an easy lemma about the strict-chromatic polynomial, and explore some immediate consequences.
\begin{lemma}\label{lem:chi-non-zero}
Let $D$ be a digraph and let $\ell(D)$ be the maximum number of arcs of a directed path of $D$ (so that $\ell=\infty$ if $D$ contains a directed cycle). Then $\chi_D^{\As}(q)=0$ for any positive integer $q\leq \ell$, and $\chi_D^{\As}(q)>0$ for any positive integer $q> \ell$. In particular, $\chi_D^{\As}\neq 0$ if and only if $D$ is acyclic. 
\end{lemma}

\begin{proof}
Let $q$ be a positive integer. If $f$ is a $q$-coloring with only strict ascents, then the colors must be strictly increasing along any directed path. Hence, $\chi_D^{\As}(q)=0$ for all $q\leq \ell$. Conversely, $\chi_D^{\As}(q)>0$ for all $q>\ell$, because the $q$-coloring $f$ which associates to each vertex $v$ the length of the longest directed path of $D$ ending at $v$, has only strict ascents.
\end{proof}

We now show that several quantities about $D$ can be obtained from $B_D(q,y,z)$. For a digraph $D$, we denote by $\acyc(D)$ the (acyclic) digraph obtained from $D$ by contracting all the cyclic arcs. Note that the vertices of $\acyc(D)$ correspond to the \emph{strongly connected components} of~$D$. 
\begin{cor}\label{cor:additional-prop}
Let $D$ be a digraph. 
\begin{compactenum}
\item[(i)] The number $\al$ of acyclic arcs of $D$ is $\deg_y(B_D(q,y,0))$. 
\item[(ii)] The number of strongly connected components of $D$ is $\deg_q([y^{\al}]B_D(q,y,0))$.
\item[(iii)] The maximal length of directed paths in $\acyc(D)$ is $\min\{q\in \PP~|~[y^{\al}]B_D(q,y,0)\neq 0\}$.
\end{compactenum}
\end{cor}

\begin{proof}
Let $\al$ be the number of acyclic arcs of $D$, and let $\be=|A|-\al$ be the number of cyclic arcs. Observe that contracting a cyclic arc of a digraph decreases the number of cyclic arcs by one, whereas contracting an acyclic arc does not decrease the number of cyclic arcs. 
Hence, the minimal number of arcs of $D$ to be contracted in order to obtain an acyclic digraph is $\be$.
Thus,~\eqref{eq:expand-contraction} and Lemma~\ref{lem:chi-non-zero} show that $\al=\deg_y(B_D(q,y,0))$, and 
$$\chi^{\As}_{\acyc(D)}(q)=[y^{\al}]B_D(q,y,0).$$
Moreover, by Lemma~\ref{lem:chi-non-zero}, the maximal length of directed paths in $\acyc(D)$ is $\min\{q\in \PP~|~\chi^{\As}_{\acyc(D)}(q)\neq 0\}$.
This proves (i) and (iii). Lastly, we claim that the number of strongly connected components $\vv(\acyc(D))$ is the degree of $\chi^{\As}_{\acyc(D)}(q)$. Indeed, the same reasoning as in the proof of Proposition~\ref{thm:existence-B} shows that for any digraph $D'=(V',A')$,
$$\chi^{\As}_{D'}(q)=\sum_{p=1}^{|V'|}c_p(D') {q \choose p},$$
where $c_p(D')$ is the number of \emph{surjective} $p$-colorings of $D'$ with only strict ascents. Moreover, when $D'$ is acyclic, $c_q(D')\neq 0$, so that the degree of $\chi^{\As}_{D'}$ is $|V'|$. Since $\acyc(D)$ is acyclic, the degree of $\chi^{\As}_{\acyc(D)}(q)$ is indeed $\vv(\acyc(D))$.
\end{proof}


\smallskip

\section{The $B$-polynomial at negative $q$, via Ehrhart theory}\label{sec:Ehrhart}
In this section, we use Ehrhart theory in order to show that the evaluations of the $B$-polynomial at negative values of $q$ have a combinatorial interpretation in terms of orientations, and to obtain a relation about planar duality.
\subsection{Ehrhart theory and general results} 
We first interpret the polynomials $\chi_D^{\As}(q)$ and $\chi^{\wAs}(q)$ as counting lattice points in a polytope.

\begin{definition}
Let $D=(V,A)$ be a digraph. The \emph{ascent polytope} of $D$ is the polytope $\De_D\subset\RR^{V}$ made of the points $(x_v)_{v\in V}$ such that 
\begin{equation*}
\forall v\in V,~ 0\leq x_v\leq 1, ~\textrm{ and } ~\forall (u,v)\in A,~x_u\leq x_v.
\end{equation*}
\end{definition}

The $q$-\emph{dilation} of a region $\De$ of $\RR^{n}$, is 
$$q\De:=\{(q\,x_1,\ldots,q\,x_n)~|~(x_1,\ldots,x_n)\in \De\}.$$
The \emph{interior} of $\De$ is denoted by $\De^\circ$.
The following lemma relates $\chi_D^{\As}(q)$ and $\chi^{\wAs}(q)$ to the $q$-dilation of the ascent polytope. 

\begin{lemma}\label{lem:chrom-as-latticepoints}
Let $D=(V,A)$ be a digraph with $n$ vertices, and let $\De_D\subset\RR^V$ be its ascent polytope.
Then, for all positive integers $q$, 
\begin{equation}
\chi_D^{\As}(q)=|(q+1)\De_D^\circ\cap\ZZ^V|, ~\textrm{ and } ~\chi_D^{\wAs}(q)=|(q-1)\De_D\cap \ZZ^V|.
\end{equation}
\end{lemma}

\begin{proof}
The points $(x_v)_{v\in V}$ in $ (q+1)\De_D^\circ\cap \ZZ^V$ are characterized by 
$$\forall v\in V, x_v\in [q], ~\textrm{ and } ~ \forall (u,v)\in A,~x_u< x_v.$$
Hence they identify with the $q$-colorings $f$ of $D$ with only strict ascents upon setting $f(v)=x_v$ for all $v\in V$. 
Similarly, the points $(x_v)_{v\in V}$ in $(q-1)\De_D \cap \ZZ^V$ identify with the $q$-colorings $f$ of $D$ with only weak ascents upon setting $f(v)=x_v+1$. 
\end{proof}

Let us now recall the main results of Ehrhart theory.

\begin{lemma}[Ehrhart's Theorem and Ehrhart-Macdonald reciprocity]
Let $\Pi\subset \RR^n$ be a polytope with integer vertices. Then, there exists a polynomial $E_\Pi(q)$, called \emph{Ehrhart polynomial} of $\Pi$, such that for any non-negative integer~$q$,
$$E_\Pi(q) = \left| q\Pi\cap \ZZ^n\right|.$$
Moreover, if the interior $\Pi^\circ$ is non-empty, then for any positive integer~$q$, 
\begin{equation}\label{eq:Ehrhart-reciprocity}
E_\Pi(-q) = (-1)^{n}\left|q\Pi^\circ \cap \ZZ^n\right|.
\end{equation}
\end{lemma}

\begin{exmp} Consider the 2-dimensional polytope $\Pi=\{(x,y)\in \RR^2 ~|~x,y\in[0,1]\}\subset \RR^2$. We have $E_\Pi(q) = (q+1)^2$, so that $E_\Pi(-q) =(q-1)^2=\left|q\Pi^\circ \cap \ZZ^2\right|$.
\end{exmp}


We now prove a key lemma, which is a straightforward extension of a result of Stanley about order polynomials~\cite[Theorem 3]{Stanley:chromatic-ordered} (the result~\cite[Theorem 3]{Stanley:chromatic-ordered} is stated in terms of posets and corresponds to \eqref{eq:acyclic-negq}). 

\begin{lemma}\label{lem:reciprocity-chrom}
Let $D=(V,A)$ be a digraph and let $q$ be a positive integer.
Then,
\begin{eqnarray}
\chi_D^{\wAs}(-q)=(-1)^{\vv\l(\acyc(D)\r)}\chi_{\acyc(D)}^{\As}(q), \label{eq:general-negq}
\end{eqnarray}
where $\acyc(D)$ is the (acyclic) digraph obtained from $D$ by contracting all the cyclic arcs.
\end{lemma}

\begin{proof}
First note that for any cyclic arc $a$ of $D$, $\chi_D^{\wAs}=\chi_{D_{/a}}^{\wAs}$. This is because for any positive integer~$q$ and any $q$-coloring $f$ with only weak ascents, all the vertices along a directed cycle of $D$ must have the same color. Therefore, $\chi_D^{\wAs}=\chi_{\acyc(D)}^{\wAs}$. Since $\acyc(D)$ is acyclic, it only remains to prove that for an acyclic digraph $D=(V,A)$,
\begin{equation}\label{eq:acyclic-negq}
\chi_D^{\As}(q)=(-1)^{|V|}\chi_D^{\wAs}(-q).
\end{equation}
Let $D$ be acyclic and let $E_D(q)$ be the Ehrhart polynomial of the ascent polytope $\De_D$. 
By Lemma~\ref{lem:chrom-as-latticepoints}, we have $\chi_D^{\wAs}(q)=|(q-1)\De_D\cap \ZZ^V|=E_D(q-1)$.
Since $D$ acyclic, we know from Lemma~\ref{lem:chi-non-zero} that $\chi_D^{\As}\neq 0$. Hence, by Lemma~\ref{lem:chrom-as-latticepoints}, the ascent polytope $\De_D$ has non-empty interior. Thus, by Ehrhart-Macdonald reciprocity~\eqref{eq:Ehrhart-reciprocity}, 
$$\chi_D^{\As}(q)=|(q+1)\De_D^\circ\cap\ZZ^V|=(-1)^{|V|}E_D(-q-1)=(-1)^{|V|}\chi_D^{\wAs}(-q).$$
\end{proof}

We now state the key result of this section, which expresses various generating functions of weak and strict-chromatic polynomials in terms of the $B$-polynomial.

\begin{thm}\label{thm:expansions-q}
Let $D=(V,A)$ be a digraph. Then, 
\begin{eqnarray}
\sum_{\substack{R\uplus S\uplus T= A \\ D^{-T}_{\setminus R}\textrm{ acyclic}}} y^{|S|}z^{|T|}\chrom_{D^{-T}_{\setminus R}}^{\wAs}(q)&\!\!=\!\!&(-1)^{|V|} \B_D(-q,1+y,1+z),\label{eq:q-1}\\
\sum_{R\uplus S\uplus T= A} y^{|S|}z^{|T|}(-1)^{\vv\l(\acyc\l(D^{-T}_{\setminus R}\r)\r)}\chrom_{\acyc\l(D^{-T}_{\setminus R}\r)}^{\As}(q)&\!\!=\!\!&(1+y+z)^{|A|}\B_D\left(-q,\frac{1+y}{1+y+z},\frac{1+z}{1+y+z}\right),\quad\quad\label{eq:q-2}\\
\sum_{\substack{R\uplus S\uplus T = A \\ D^{-T}_{/R}\textrm{ acyclic}}} y^{|S|}z^{|T|}(-1)^{\vv\l(D^{-T}_{/R}\r)}\chrom_{D^{-T}_{/R}}^{\wAs}(q) &\!\!=\!\!&\B_D(-q,y,z),\label{eq:q-3}\\
\sum_{R\uplus S\uplus T = A} y^{|S|}z^{|T|}(-1)^{\vv\l(\acyc\l(D^{-T}_{/R}\r)\r)}\chrom_{\acyc\l(D^{-T}_{/R}\r)}^{\As}(q)&\!\!=\!\!&(1+y+z)^{|A|}\B_D\left(-q,\frac{y}{1+y+z},\frac{z}{1+y+z}\right).\quad\quad\label{eq:q-4}
\end{eqnarray}
\end{thm}

\begin{proof}
Equations (\ref{eq:q-1}-\ref{eq:q-4}) follow respectively from Equations (\ref{eq:<delete}-\ref{eq:leq-contract}) by applying Lemma~\ref{lem:reciprocity-chrom}.
\end{proof}

\begin{remark}\label{rk:Potts-compatible-colorings}
Note that \eqref{eq:q-1} is particularly interesting from a combinatorial point of view because its left-hand side has no  negative signs.
We now point out a relation between this equation and the seminal result given by Stanley in \cite{Stanley:acyclic-orientations} about the chromatic polynomial of a graph. 
The result in \cite{Stanley:acyclic-orientations} can be stated as follows: for a graph $G$ and for a positive integer~$q$, the evaluation $(-1)^{|V|}\chi_G(-q)$ gives the number of pairs $(D,f)$, where $D$ is an acyclic orientation of $G$ and $f$ is a $q$-coloring of $D$ without strict descent. With our notation, this can be written as 
\begin{equation}\label{eq:Stanley}
\sum_{D\in\ori(G)}\chi_D^{\geq}(q)=(-1)^{|V|}\chi_G(-q).
\end{equation}
One of the ways to recover \eqref{eq:Stanley} from  \eqref{eq:q-1} is to set $z=y$ in \eqref{eq:q-1} and use \eqref{eq:PottsThree}. This gives
\begin{equation}\label{eq:PStanley}
\sum_{\substack{R\uplus S\uplus T= A \\ D^{-T}_{\setminus R}\textrm{ acyclic}}} y^{|A\setminus R|}=(-1)^{|V|} P_G(-q,1+y),
\end{equation}
where $G=\uD$. Extracting the coefficient of $y^{|A|}$ gives \eqref{eq:Stanley} (since $\chi_G(-q)=[y^{|A|}]P_G(-q,y)$). Note that \eqref{eq:PStanley} actually gives an interpretation for each coefficient of  $P_G(-q,1+y)$.
\end{remark}


\subsection{Evaluation at $q=-1$ and generating functions for acyclic and totally cyclic deformations of $D$.}
Specializing Theorem~\ref{thm:expansions-q} to $q=1$ shows that $B_D(-1,y,z)$ contains several generating functions for the acyclic and totally cyclic digraphs obtained from $D$ by deletions, contractions, or reorientations.

\begin{thm}\label{thm:q=-1}
Let $D=(V,A)$ be a digraph. Then, 
\begin{eqnarray}
\sum_{\substack{R\uplus S\uplus T= A \\ D^{-T}_{\setminus R}\textrm{ acyclic}}}\!\!\! y^{|S|}z^{|T|}&\!\!=\!\!&(-1)^{|V|} \B_D(-1,1+y,1+z),\label{eq:q1-1}\\
\sum_{\substack{R\uplus S\uplus T= A\\ D^{-T}_{\setminus R}\textrm{ totally cyclic}}}\!\!\! y^{|S|}z^{|T|}(-1)^{\comp\l(D^{-T}_{\setminus R}\r)}&\!\!=\!\!&(1+y+z)^{|A|}\B_D\left(-1,\frac{1+y}{1+y+z},\frac{1+z}{1+y+z}\right),\label{eq:q1-2}\\
\sum_{\substack{R\uplus S\uplus T = A \\ D^{-T}_{/R}\textrm{ acyclic}}}\!\!\! y^{|S|}z^{|T|}(-1)^{\vv\l(D^{-T}_{/R}\r)} &\!\!=\!\!&\B_D(-1,y,z),\label{eq:q1-3}\\
\sum_{\substack{R\uplus S\uplus T = A \\ D^{-T}_{/R}\textrm{ totally cyclic}}}\!\!\! y^{|S|}z^{|T|}&\!\!=\!\!&(-1)^{\comp(D)}(1+y+z)^{|A|}\B_D\left(-1,\frac{y}{1+y+z},\frac{z}{1+y+z}\right).
\quad\label{eq:q1-4}
\end{eqnarray}
\end{thm}

\begin{proof}
Observe that $\chi_D^{\wAs}(1)=1$ for any digraph $D$. Hence, plugging $q=1$ in~\eqref{eq:q-1} and~\eqref{eq:q-3} gives~\eqref{eq:q1-1} and~\eqref{eq:q1-3} respectively. Observe also that $\chi_D^{\As}(1)=\ONE_{D\textrm{ has no arcs}}$. Hence, $\chi_{\acyc(D)}^{\As}(1)=\ONE_{D\textrm{ is totally cyclic}}$.
Moreover, if $D$ is totally cyclic, $\vv(\acyc(D))=\comp(D)$. Thus, plugging $q=1$ in~\eqref{eq:q-2} and~\eqref{eq:q-4} gives~\eqref{eq:q1-2} and~\eqref{eq:q1-4} respectively.
\end{proof}

\begin{remark} Observe that Theorem~\ref{thm:q=-1} could also be derived directly from the Equations (\ref{eq:<delete}-\ref{eq:leq-contract}) using the following identity:
\begin{equation}\label{eq:indicators}
\chi^{\As}_D(-1)=(-1)^{|V|}\ONE_{D \textrm{ acyclic}},~~\textrm{and}~~\chi^{\wAs}_D(-1)=(-1)^{\comp(D)}\ONE_{D \textrm{ totally cyclic}}.
\end{equation}
The identity~\eqref{eq:indicators}, in turn, follows from Lemma~\ref{lem:reciprocity-chrom}, along with the facts that $\chi_D^{\wAs}(1)=1$, and $\chi_D^{\As}(1)=\ONE_{D\textrm{ has no arcs}}$.
\end{remark}

Among the results of Theorem~\ref{thm:q=-1}, Equations~\eqref{eq:q1-1} and~\eqref{eq:q1-4} are especially nice because they do not involve any sign on their left-hand side. Equation~\eqref{eq:q1-1} gives the bivariate generating function of acyclic digraphs obtained from $D$ by deleting and reorienting some arcs. 
Let us state the specialization of~\eqref{eq:q1-1} corresponding to reorientations.
\begin{cor}\label{cor:GFacyc1}
The generating function of acyclic reorientations of $D$, counted according to the number of reoriented arcs is
\begin{equation}\label{eq:GF-acyclic-reorientation}
\sum_{\substack{S \subseteq A \\ D^{-S}\textrm{ acyclic}}} y^{|S|}= (-1)^{|V|}[z^{|A|}]\B_D(-1,yz,z).
\end{equation}
\end{cor}


\begin{example}
Consider the digraph $D$ represented on the left of Figure~\ref{fig:reorientation}. We can compute
$$B_{D}(q,y,z)=q+q(q-1)(y^2+z^2+yz)+\frac{q(q-1)(q-2)}{6}(y^3+z^3+2yz(y+z)),$$
so that $\ds -[z^{3}]\B_{D}(-1,yz,z)=1+2y+2y^2+y^3$. 
Looking at Figure~\ref{fig:reorientation}, we see that this matches the generating function $\sum_{\substack{S \subseteq A,~ D^{-S}\textrm{ acyclic}}} y^{|S|}$.
\end{example}

\fig{width=.5\linewidth}{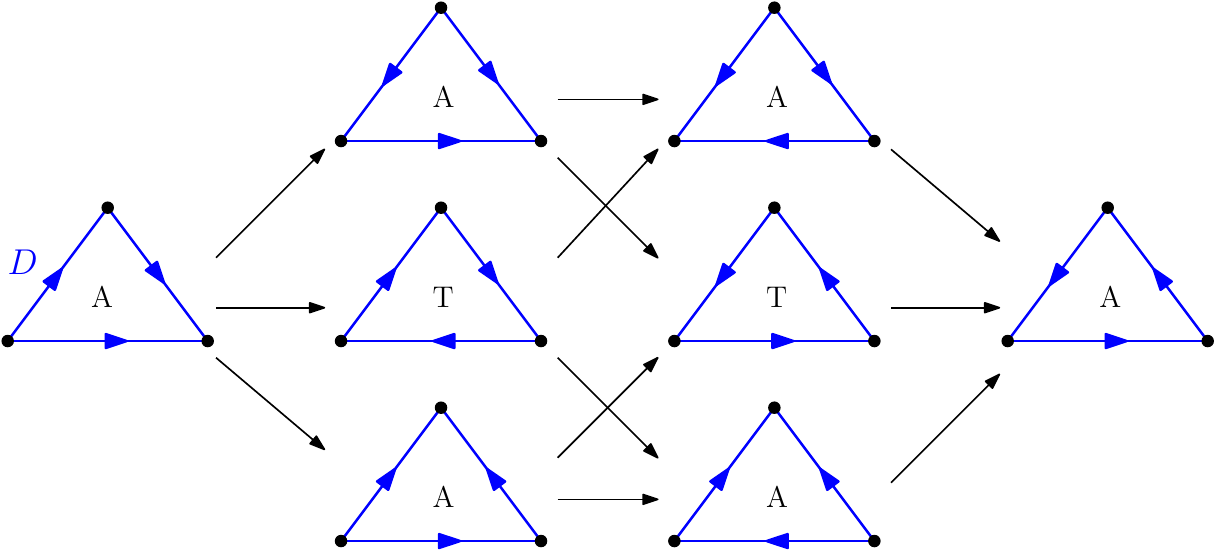}{A digraph $D$ and its reorientations (organized according to the ``graph of arc-flips''). The 6 acyclic reorientations are indicated by a letter $A$, while the 2 totally cyclic reorientations are marked by a letter $T$. We see that $\ds\sum_{S \subseteq A,~ D^{-S}\textrm{ acyclic}} y^{|S|}=1+2y+2y^2+y^3$, and $\ds\sum_{S \subseteq A,~ D^{-S}\textrm{ totally cyclic}} y^{|S|}=y+y^2$.}

Next we state the specialization of~\eqref{eq:q1-1} corresponding to acyclic sub-orientations.
\begin{cor}\label{cor:GFacyc2}
The generating function of acyclic subgraphs of $D$, counted according to the number of arcs is 
\begin{equation}\label{eq:GF-acyclic-deletion}
\sum_{\substack{S \subseteq A \\ D_{\setminus S}\textrm{ acyclic}}} y^{|A\setminus S|}=(-1)^{|V|} \B_D(-1,1+y,1).
\end{equation}
\end{cor}


\begin{example}
For the digraph $D$ represented on the left of Figure~\ref{fig:subdigraph2} we have
$$B_{D}(q,y,z)=q+3q(q-1)yz+\frac{q(q-1)(q-2)}{2}yz(y+z),$$
so that $\ds (-1)^{|V|} \B_D(-1,1+y,1)=1+3y+3y^2$. 
Looking at Figure~\ref{fig:subdigraph2}, we see that this matches the generating function $\ds \sum_{S \subseteq A, ~ D_{\del S}\textrm{ acyclic}} y^{|A\setminus S|}$.
\end{example}

\fig{width=.5\linewidth}{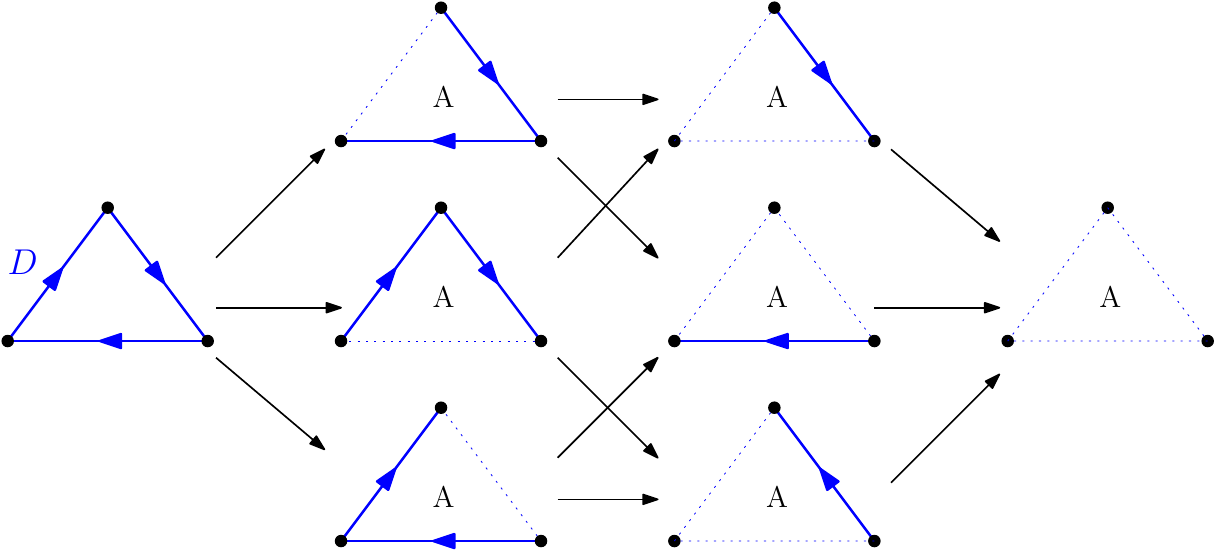}{A digraph $D$ and its subgraphs. The 7 acyclic subgraphs are indicated by a letter $A$. We see that $\ds\sum_{S \subseteq A,~ D_{\del S}\textrm{ acyclic}} y^{|A\setminus S|}=1+3y+3y^2$.}


Next we state the specialization of~\eqref{eq:q1-4} corresponding to totally cyclic reorientations.
\begin{cor}\label{cor:GFcyc1}
The generating function of totally cyclic reorientations of $D$, counted according to the numbers of reoriented arcs is
\begin{equation}\label{eq:GF-cyclic-reorientation}
\sum_{\substack{S \subseteq A \\ D^{-S}\textrm{ totally cyclic}}} y^{|S|}= (-1)^{\comp(D)}(1+y)^{|A|}\B_D\left(-1,\frac{y}{1+y},\frac{1}{1+y}\right).
\end{equation}
\end{cor} 


\begin{example}
For the digraph $D$ represented on the left of Figure~\ref{fig:reorientation}, we find
$$(-1)^{\comp(D)}(1+y)^{|A|}\B_D\left(-1,\frac{y}{1+y},\frac{1}{1+y}\right)=y+y^2.$$
Looking at Figure~\ref{fig:reorientation}, we see that this matches the generating function $\ds \sum_{S \subseteq A,~ D^{-S}\textrm{ totally cyclic}} y^{|S|}$.
\end{example}

Lastly we state the specialization of~\eqref{eq:q1-4} corresponding to totally cyclic contractions.
\begin{cor} \label{cor:GFcyc2}
The generating function of totally cyclic contractions of $D$, counted according to the number of arcs is 
\begin{equation}\label{eq:GF-cyclic-contraction}
\sum_{\substack{S \subseteq A \\ D_{/ S}\textrm{ totally cyclic}}} y^{|A\setminus S|}=(-1)^{\comp(D)}(1+y)^{|A|}\B_D\left(-1,\frac{y}{1+y},0\right).
\end{equation}
\end{cor}

\begin{example}
For the digraph $D$ represented on the left of Figure~\ref{fig:contraction}, we find 
$$(-1)^{\comp(D)}(1+y)^{|A|}\B_D\left(-1,\frac{y}{1+y},0\right)(-1)^{\comp(D)}=y^2+3y+1.$$
Looking at Figure~\ref{fig:contraction}, we see that this matches the generating function $\ds \sum_{S \subseteq A,~ D_{/S} \textrm{ totally cyclic}} y^{|A\setminus S|}$.
\end{example}

\fig{width=.5\linewidth}{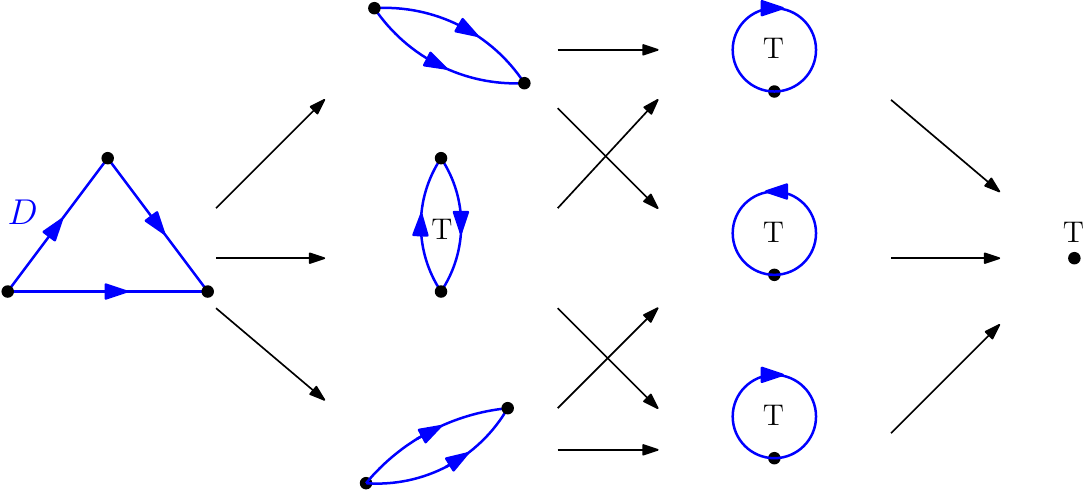}{A digraph $D$ and its contractions. The 5 totally cyclic contractions are indicated by a letter $T$. We see that $\ds\sum_{S \subseteq A,~ D_{\del S}\textrm{ totally cyclic}} y^{|A\setminus S|}=y^2+3y+1$.}


We end this subsection with a question. Note that setting $y=-1$ in~\eqref{eq:deletion-expand} gives 
$$\ds \sum_{S\subseteq A}(-1)^{|A\setminus S|}\chi_{D_{\del S}}^{\As}(q)=\chi_D^{\wAs}(q).$$
Similarly, from~\eqref{eq:deletion-expand2} we get 
$\ds \sum_{S\subseteq A}(-1)^{|A\setminus S|}\chi_{D_{\del S}}^{\wAs}(q)=\chi_D^{\As}(q).$
These identities can also be obtained by a direct inclusion-exclusion argument. Now, setting $q=-1$ and using~\eqref{eq:indicators} gives the following intriguing result.
\begin{prop}\label{prop:mysterious}
For any digraph $D=(V,A)$,
\begin{eqnarray}
(-1)^{|V|-\comp(D)}\sum_{\substack{S\subseteq A \\ D_{\del S} \textrm{ acyclic}}}(-1)^{|A\setminus S|}&=&\ONE_{D\textrm{ totally cyclic}},\label{eq:myster-chi}\\
\sum_{\substack{S\subseteq A \\ D_{\del S} \textrm{ totally cyclic}}}(-1)^{|A\setminus S|+\comp(D_{\del S})-|V|}&=&\ONE_{D\textrm{ acyclic}}.\label{eq:myster-chi2}
\end{eqnarray}
\end{prop}


\begin{question}\label{question:mysterious-identities-1}
Is it possible to prove the identities~\eqref{eq:myster-chi} and~\eqref{eq:myster-chi2} by a direct combinatorial argument (for instance, by using a ``sign reversing involution'')?
Note that~\eqref{eq:myster-chi2} is obvious when $D$ is acyclic (only 1 totally cyclic subgraph). It is also easy to prove~\eqref{eq:myster-chi} when $D$ is not totally cyclic. Indeed this can be done by considering the involution which ``flips'' a given acyclic arc of $D$ (adding it if is absent, removing it if it is present). But the other cases seem more challenging. 
\end{question}

\subsection{A planar duality relation}
In this subsection we establish a relation between the $B$-polynomial of a planar digraph and its dual.
Recall that a digraph $D$ is \emph{planar} if it can be drawn in the plane without arc crossings. Let $D$ be a connected planar digraph. A \emph{dual digraph} is a digraph $D^*$ obtained by choosing a planar drawing of $D$, and then placing a vertex of $D^*$ in each face of $D$, and drawing an arc $a^*$ of $D^*$ across each arc $a$ of $D$ with $a^*$ oriented from the left of $a$ to the right of $a$. See Figure~\ref{fig:duality}. If $D$ is a disconnected planar digraph, then a dual is obtained by applying the preceding procedure to each connected component of $D$.

\fig{width=.8\linewidth}{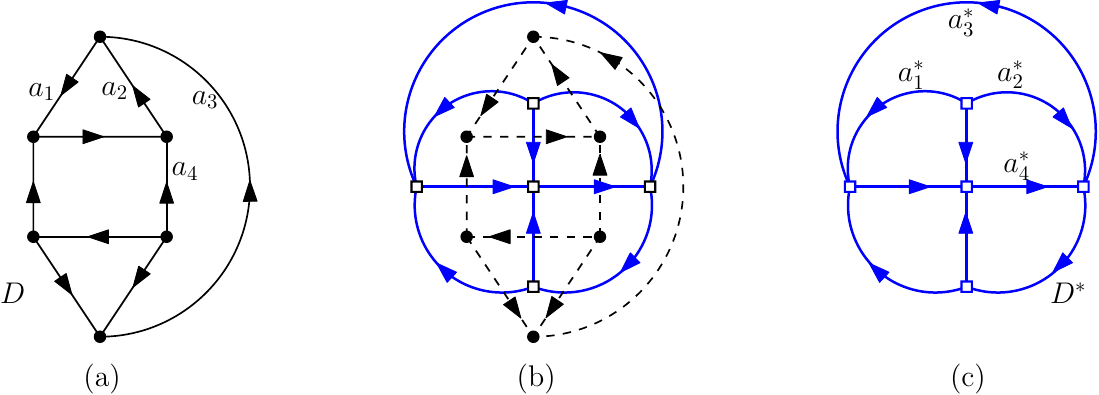}{(a) A planar digraph. (b) Constructing the dual digraph $D^*$, by drawing an arc of $D^*$ across each arc of $D$. (c) The dual digraph $D^*$.}

\begin{thm}\label{thm:duality}
Let $D=(V,A)$ be a planar digraph, and let $D^*=(V^*,A^*)$ be a dual digraph. Then the polynomials $B_{D}(-1,y,z)$ and $B_{D^*}(-1,y,z)$ are related by the following change of variables:
\begin{equation}\label{eq:duality}
B_{D^*}(-1,y,z)=(-1)^{\comp(D)-|V|}(1-y-z)^{|A|}B_D\l(-1,\frac{1-y}{1-y-z},\frac{1-z}{1-y-z}\r).
\end{equation}
\end{thm}

\begin{proof} 
We will explain how to obtain~\eqref{eq:duality} by comparing~\eqref{eq:q1-1} for $D$ with~\eqref{eq:q1-4} for $D^*$. 
Let $S\subseteq A$ and let $S^*\subseteq A^*$ be the set of arcs dual to $S$. Note that the dual of $D^{-S}$ is ${D^*}^{-S^*}$, and the dual of $D_{\del S}$ is $D^*_{/S^*}$ (because the dual of deleting an arc is contracting the dual arc). 
Moreover, a planar digraph $D'$ is acyclic if and only if its dual ${D'}^*$ is totally cyclic (because the dual of an acyclic arc is a cyclic arc).
Thus using~\eqref{eq:q1-1} for $D$ and~\eqref{eq:q1-4} for $D^*$ gives 
\begin{eqnarray*}
(-1)^{|V|}\B_D(-1,1+y,1+z)&=&\sum_{\substack{R\uplus S\uplus T= A \\ D^{-T}_{\setminus R}\textrm{ acyclic}}}\!\!\! y^{|S|}z^{|T|}\\
&=&\sum_{\substack{R^*\uplus S^*\uplus T^*= A^* \\ {D^*}^{-T^*}_{/ {R^*}}\textrm{ totally cyclic}}}\!\!\! y^{|S|}z^{|T|}\\
&=&(-1)^{\comp(D)}(1+y+z)^{|A|}\B_{D^*}\left(-1,\frac{y}{1+y+z},\frac{z}{1+y+z}\right),\\
\end{eqnarray*}
which is equivalent to~\eqref{eq:duality}.
\end{proof}


\begin{remark} 
Recall the classical duality relation for the Tutte polynomial of a planar graph $G=(V,E)$:
\begin{equation}\label{eq:duality-Tutte}
T_{G^*}(x,y)=T_{G}(y,x). 
\end{equation}
Using~\eqref{eq:Potts=Tutte} (and the Euler relation $|V|+|V^*|=|A|+2\comp(D)$), this translates into the following relation for the Potts polynomial:
\begin{equation}\label{eq:duality-Pclassical}
P_{G^*}(q,y)=\frac{\l((q-1)y+1\r)^{|A|}}{q^{|V|-\comp(G)}}P_G\l(q,\frac{y-1}{(1-q)y-1}\r).
\end{equation}
Hence, using~\eqref{eq:PottsThree} we obtain another bivariate duality relation for the $B$-polynomial of a planar digraph $D=(V,A)$:
\begin{equation}\label{eq:duality-Bclassical}
\B_{D^*}(q,y,y)=\frac{\l((q-1)y+1\r)^{|A|}}{q^{|V|-\comp(G)}}\B_D\l(q,\frac{y-1}{(1-q)y-1},\frac{y-1}{(1-q)y-1}\r).
\end{equation}
Although a duality relation holds for the two bivariate specializations $y=z$ and $q=-1$ of the $B$-polynomial, there cannot be a duality relation for the trivariate $B$-polynomial. 
Indeed, the dual of any tree is a digraph with only loops, hence having $B$-polynomial equal to 1. But since the $B$-polynomial detects the length of directed paths, there must be infinitely many $B$-polynomials of trees (even after renormalizing by any prefactor depending on the number of vertices, arcs, and connected components). However, the previous argument does not exclude that a certain refinement of the $B$-polynomial could have a full duality relation.
\end{remark}

\begin{remark}
The invariant $(-1)^{c(D)}B_D(-1,y,z)$ appearing in Theorem \ref{thm:duality} is an \emph{oriented-matroid invariant}: it only depends on the oriented-matroid associated with $D$ (as mentioned earlier, this is not the case for $B_D(q,y,z)$). This fact can be proved by a direct combinatorial argument using the so-called \emph{Whithney flips}. In \cite{Awan-Bernardi:A-poly} we prove that this invariant is actually a specialization of an oriented-matroid invariant, called \emph{$A$-polynomial}, which is analoguous to the $B$-polynomial.
\end{remark}

\smallskip

\section{Tutte polynomials and their evaluations} \label{sec:Tutte-eval} 
In this section, we reinterpret some of the results of the previous sections (those about the specialization $B_D(q,y,1)$) in terms of the invariants $T^{(1)}_D$ and $T^{(2)}_D$ of mixed graphs, and establish some links with the classical theory of the Tutte polynomial.


Let us first express the invariants $T^{(1)}$ and $T^{(2)}$ in terms of the strict and weak chromatic polynomials. Using~\eqref{eq:deletion-expand} and~\eqref{eq:deletion-expand2} we get two expressions for each invariant:
\begin{eqnarray}
T_D^{(1)}(x,y)&\!\!=\!\!&\frac{y^{|E|}}{(x-1)^{\comp(D)}(y-1)^{|V|}}\sum_{S\subseteq A}\l(\frac{1-y}{y}\r)^{|A\setminus S|}\chi_{D_{\del S}}^{\As}((x-1)(y-1)),\label{eq:T1<}\\
T_D^{(1)}(x,y)&\!\!=\!\!&\frac{y^{|E|-|A|}}{(x-1)^{\comp(D)}(y-1)^{|V|}}\sum_{S\subseteq A}(y-1)^{|A\setminus S|}\chi_{D_{\del S}}^{\wAs}((x-1)(y-1)),\label{eq:T1<=}\\
T_D^{(2)}(x,y)&\!\!=\!\!&\frac{(y/2)^{|E|}}{(x-1)^{\comp(D)}(y-1)^{|V|}}\sum_{\vec{D}=(V,\vec{E})\in \ori(D)}\,\sum_{S\subseteq \vec{E}}\l(\frac{2(1-y)}{y}\r)^{|\vec{E}\setminus S|} \chi_{\vec{D}_{\del S}}^{\As}((x-1)(y-1)),\quad \label{eq:T2<}\\
T_D^{(2)}(x,y)&\!\!=\!\!&\frac{(1-y/2)^{|E|}}{(x-1)^{\comp(D)}(y-1)^{|V|}}\sum_{\vec{D}=(V,\vec{E})\in \ori(D)}\,\sum_{S\subseteq \vec{E}}\l(\frac{y-1}{1-y/2}\r)^{|\vec{E}\setminus S|} \chi_{\vec{D}_{\del S}}^{\wAs}((x-1)(y-1)).\quad \label{eq:T2<=}
\end{eqnarray}

We now prove the polynomiality of $T_D^{(1)}$ and $T_D^{(2)}$. 
\begin{prop}\label{prop:polynomiality}
For any digraph $D=(V,A)$, the invariants $T^{(1)}_D(x,y)$ and $T^{(2)}_D(x,y)$ defined by~\eqref{eq:defT1} and~\eqref{eq:defT2} are polynomials in $x$ and $y$.
\end{prop}

\begin{proof}
For a digraph $D$, we write $\wchrom_{D}^{\As}(q)=q^{-\comp(D)}\chrom_{D}^{\As} (q)=[y^{|A|}]q^{-\comp(D)}B_{D}(q,y,1)$. 
By Proposition~\ref{prop:easy}\eqref{it9}, $q^{\comp(D)}$ divides $B_{D}(q,y,z)$, hence $\wchrom_{D}^{\As}(q)$ is a polynomial in $q$.
Moreover,~\eqref{eq:T1<} gives
$$T^{(1)}_D(x,y)=\sum_{S\subseteq A}(-1)^{|A\setminus S|}(x-1)^{\comp(D_{\del S})-\comp(D)}y^{|E|-|A\setminus S|}(y-1)^{|A\setminus S|+\comp(D_{\del S})-|V|}\wchrom_{D_{\del S}}^{\As}((x-1)(y-1)).$$
The exponent of $(x-1)$ in the above sum is clearly non-negative. 
The exponent of $y$ is non-negative if $D_{\del S}$ is acyclic, which we can assume by Lemma~\ref{lem:chi-non-zero}. 
The exponent of $(y-1)$ is also non-negative since it corresponds to the \emph{nullity} (a.k.a. cyclomatic number) of the graph underlying $D_{\setminus S}$. Hence $T^{(1)}_D(x,y)$ is a polynomial. The same argument starting from~\eqref{eq:T2<} shows that $T_D^{(2)}(x,y)$ is a polynomial.
\end{proof}

 

Equations~\eqref{eq:T1<} and~\eqref{eq:T2<} simplify greatly for $y=0$ and give
\begin{equation*}
T^{(1)}_D(x,0)=T^{(2)}_D(x,0)=\frac{(-1)^{|V|}}{(x-1)^{\comp(D)}}\sum_{\vec{D}\in \ori(D)}\chi_{\vec{D}}^{\As}(1-x).
\end{equation*}
Now, using~\eqref{eq:acyclic-negq} gives
\begin{equation}\label{eq:Tutte-y=0bis}
(-1)^{|V|}\chi_D(1-x)=(x-1)^{\comp(D)}T^{(1)}_D(x,0)=(x-1)^{\comp(D)}T^{(2)}_D(x,0)=\sum_{\substack{\vec{D}\in \ori(D)\\ \vec{D}\textrm{ acyclic}}}\chi_{\vec{D}}^{\wAs}(x-1).
\end{equation}
As a consequence, we obtain the following result.
\begin{cor}[\cite{Beck:strict-chromatic-poly-digraph}]\label{cor:T20}
For any mixed graph $D$, 
\begin{equation}\label{eq:T20}
(-1)^{|V|}\chi_D(-1)=T^{(1)}_D(2,0)=T^{(2)}_D(2,0)=\#\textrm{ acyclic complete orientations of }D.
\end{equation}
More generally, 
$$(-1)^{|V|}\chi_D(-q)=\sum_{f:V\to [q]}\ga_f,$$ 
where the sum is over all the $q$-colorings of $D$, and $\ga_f$ is the number of acyclic complete orientations $\vec{D}$ of $D$ such that there is no strict descent of colors along the arcs of $\vec{D}$.
\end{cor}

\begin{proof}
Setting $x=2$ in~\eqref{eq:Tutte-y=0bis} immediately gives~\eqref{eq:T20} since $\chi_{\vec{D}}^{\wAs}(1)=1$. Setting $x=q+1$ in~\eqref{eq:Tutte-y=0bis} gives
$$(-1)^{|V|}\chi_D(-q)=\sum_{\substack{\vec{D}\in \ori(D)\\ \vec{D}\textrm{ acyclic}}}\chi_D^{\wAs}(q).$$
Moreover the right-hand side can be interpreted as the number of pairs $(\vec{D},f)$, where $\vec{D}$ an acyclic complete orientations of $D$, and $f$ is a $q$-coloring such that there is no strict descent of colors along the arcs of $\vec{D}$.
\end{proof}

Corollary~\ref{cor:T20} was first proved by Stanley for unoriented graphs~\cite{Stanley:acyclic-orientations}, and by Beck, Bogart and Pham for mixed graphs~\cite{Beck:strict-chromatic-poly-digraph}\footnote{Only the invariant $\chi_D$ is studied in~\cite{Beck:strict-chromatic-poly-digraph}, and not $T^{(1)}_D$, $T^{(2)}_D$.}.\\

Similarly,~\eqref{eq:T2<=} simplifies for $y=2$:
\begin{equation}\label{eq:Tutte2-y=2}
T^{(2)}_D(x,2)=\frac{1}{(x-1)^{\comp(D)}}\sum_{\vec{D}\in \ori(D)}\chi_{\vec{D}}^{\wAs}(x-1).
\end{equation}
In particular, setting $x=0$ in~\eqref{eq:Tutte2-y=2}, and using~\eqref{eq:indicators} gives the following result dual to~\eqref{eq:T20}.
\begin{cor}\label{cor:T02}
For any mixed graph $D$, 
\begin{equation}\label{eq:T02}
T^{(2)}_D(0,2)=\#\textrm{ totally cyclic complete orientations of }D.
\end{equation}
\end{cor}
This generalizes to mixed graphs the result established by Las Vergnas for unoriented graphs~\cite{LasVergnas:Tutte(02)}.

\begin{example}
Recall that the polynomials $T_{D''}^{(1)}(x,y)$ and $T_{D''}^{(1)}(x,y)$ for the mixed graph $D''$ of Figure~\ref{fig:exp-intro} (considered as a mixed graph with one unoriented edge) are given in~\eqref{eq:exp-intro-T1} and~\eqref{eq:exp-intro-T2}. 
We find $T^{(1)}_{D''}(2,0)=T^{(2)}_{D''}(2,0)=T^{(2)}_{D''}(0,2)=1$ reflecting the fact that there is 1 complete orientation of $D$ which is acyclic and 1 which is totally cyclic.
\end{example}


We now generalize~\eqref{eq:T20}.
Translating~\eqref{eq:GF-acyclic-deletion} in terms of $T^{(1)}_D$ gives the following result.
\begin{thm}\label{thm:T1-acyclic}
For any mixed graph, the polynomial $T_D^{(1)}(x,y)$ contains the generating function of the acyclic subgraphs of $D$, counted according to the number of arcs:
\begin{equation}\label{eq:T1-acyclic}
(y+1)^{|E|+\comp(D)-|V|} T^{(1)}_D\l(y+2,\frac{y}{y+1}\r) = \sum_{S\subseteq A,~D_{\setminus S} \textrm{ acyclic}}y^{|E|-|A\setminus S|}.
\end{equation}
\end{thm}

\begin{example}
For the mixed graph $D''$ of Figure~\ref{fig:exp-intro}, one finds 
$(y+1) T^{(1)}_{D''}\l(y+2,\frac{y}{y+1}\r)=y^3+4y^2+5y+1$ reflecting the fact that the number of acyclic subgraphs of $D$ with 0 arc (resp. 1 arc, 2 arcs, 3 arcs) is 1 (resp. 4, 5, 1).
\end{example}

Note that the special case $y=0$ of Theorem~\ref{thm:T1-acyclic} is~\eqref{eq:T20}. 
The case of Theorem~\ref{thm:T1-acyclic} corresponding to unoriented graphs was first proved by Gessel and Sagan~\cite{Gessel:Tutte-poly+DFS}.
This result for undirected graphs was also rediscovered as part of the theory of \emph{fourientations} for the Tutte polynomial developed in~\cite{Backman:partial-orientations,Backman:fourientation-evaluations,Backman:fourientation-activities}. Recall that in a \emph{fourientation} of a graph $G$, the edges of $G$ can be oriented in either direction (1-way edges), in both direction (2-way edges), or in no direction (0-way edges). So \emph{fourientations} of a graph $G$ are naturally identified with the subgraphs of the digraph $\orient{G}$.

The paper~\cite{Backman:partial-orientations} also contains the following result for a graph $G$: 
\begin{equation}\label{eq:fromBackman}
(y+1)^{|V|-\comp(G)} T_G(y/(y+1),y+2)=\sum_{\substack{\textrm{totally cyclic fourientations of }G \\ \textrm{without 0-way edge}}}y^{\# \textrm{2-way edges}}.
\end{equation}
This result can be recovered from~\eqref{eq:q1-4} for the digraph $D=\orient{G}$ as follows. Setting $z=-1$ in~\eqref{eq:q1-4} gives 
$$\sum_{\substack{R\uplus S\uplus T = A \\ D^{-T}_{/R}\textrm{ totally cyclic}}}\!\!\! y^{|S|}(-1)^{|T|} = (-1)^{\comp(D)}y^{|A|}\B_D\l(-1,-1/y,1\r),$$
which translates into
\begin{equation}\label{eq:fromq1-4}
y^{|A|-|E|}(y-1)^{|V|-\comp(D)}T_{D}^{(1)}\l(\frac{y-2}{y-1},y\r)=\sum_{\substack{R\uplus S\uplus T = A \\ D^{-T}_{/R}\textrm{ totally cyclic}}}\!\!\! y^{|S|}(-1)^{|R|}.
\end{equation}

Now suppose $D=(V,A)=\orient{G}$ is a graph. To a partition $R\uplus S\uplus T= A$, we associate the fourientation where an edge $e=\{a,-a\}$ of $G$ is oriented 1-way if one of the arcs $a,-a$ is in $S$ and the other is in $T$, and oriented 2-way in all the other cases. Observe that $D^{-T}_{/R}$ is totally cyclic if and only if the associated fourientation of $G$ is totally cyclic. Moreover, one can check that the 7 different configurations where the edge $e$ would be 2-ways in the associated fourientation have a total contribution of $y^2-2y$. Thus,~\eqref{eq:fromq1-4} gives
$$y^{|E|}(y-1)^{|V|-\comp(D)}T^{(1)}_{\orient{G}}\l(\frac{y-2}{y-1},y\r)=\sum_{\substack{\textrm{totally cyclic fourientations} \\ \textrm{without 0-way edge}}}y^{\# \textrm{1-way edges}}(y^2-2y)^{\# \textrm{2-way edges}},$$
which is equivalent to~\eqref{eq:fromBackman}.


We end this section by mentioning two mysterious identities which could deserve further investigation. We know that for any graph $G=(V,E)$, $T^{(1)}_{\orient{G}}(0,2)=T_G(0,2)$ is the number of totally cyclic orientations of $G$.
Hence, specializing~\eqref{eq:T1<=} to $(x,y)=(0,2)$, using~\eqref{eq:indicators} gives
\begin{equation}\label{eq:myster4}
2^{-|E|}\sum_{\substack{S\subseteq \orient{E} \\ \orient{G}_{\setminus S} \textrm{ totally cyclic}}}(-1)^{\comp(D_{\del S})-\comp(D)}=\#\textrm{ totally cyclic orientations of }G,
\end{equation}
where $\orient{E}$ is the arc-set of $\orient{G}$.
Similarly, specializing~\eqref{eq:T1<} to $(x,y)=(0,2)$, using~\eqref{eq:indicators} gives
\begin{equation}\label{eq:myster5}
2^{|E|}(-1)^{|V|-\comp(D)}\sum_{\substack{S\subseteq \orient{E} \\ \orient{G}_{\setminus S} \textrm{ acyclic}}}(-1/2)^{|\orient{E}\setminus S|}=\#\textrm{ totally cyclic orientations of }G.
\end{equation}

\begin{remark}
Equations~\eqref{eq:myster4} and~\eqref{eq:myster5} can be interpreted in terms of the fourientations of $G$. 
Equation~\eqref{eq:myster5} can be written as 
$$2^{|E|}(-1)^{|V|-\comp(D)}\sum_{\textrm{acyclic fourientation of $G$ without 2-way edges}}(-1/2)^{\# \textrm{ 1-way edge}}(1/4)^{\#\textrm{ 0-way edges}}=T_G(0,2).$$
and this follows easily from the results in~\cite{Backman:partial-orientations}.
As for~\eqref{eq:myster4}, Sam Hopkins communicated to us the following alternative proof. By setting $\lambda=1/2$, $\xi=-2$, $x=-2$, and $y=-1/2$ in Kung's convolution-multiplication formula~\cite[Identity 3]{Kung:convolutionTutte}, one gets 
\begin{equation}\label{eq:Kung-special} 
2^{-|E|}\sum_{S \subseteq E(G)} (-1)^{\comp(G_{\del S})-\comp(G)}2^{|V|-\comp(G_{\del S})} T_{G_{\del S}}(1/2,3)=T_G(0, 2),
\end{equation}
by using the fact that for any graph $H=(U,F)$, $T_H(-1,1/2)=(-1)^{|U|-\comp(H)}2^{|U|-\comp(H)-|F|}$. 
Moreover, by~\cite{Backman:fourientation-activities}, $2^{|U|-\comp(H)} T_{H}(1/2,3)$ is the number $\al_H$ of totally cyclic fourorientations of $H$ without 0-way edges.
Plugging this result in~\eqref{eq:Kung-special} gives
$$2^{-|E|} \sum_{S \subseteq E} (-1)^{\comp(G_{\del S})-\comp(G)} \al_{G_{\del S}}=T_G(0,2),$$
which is equivalent to~\eqref{eq:myster4}.
Although this proof bypasses our use of Ehrhart theory, it is rather indirect and one could hope for a more transparent combinatorial explanation.
\end{remark}

\begin{question}\label{question:mysterious-identities}
Is it possible to prove~\eqref{eq:myster4} and~\eqref{eq:myster5} by a direct combinatorial argument?
\end{question}

\smallskip
\section{A quasisymmetric function generalization of the $B$-polynomial}\label{sec:quasisym}
In this section, we study a refinement of the $B$-polynomial with infinitely many variables $\{x_i\}_{i\in\PP}$. This invariant counts colorings by ascents and descents, but also records the number of vertices colored $i$ for all $i\in \PP$. As will be clear from the definition, this invariant is quasisymmetric in the variables $\{x_i\}_{i\in\PP}$, and we study its expansion in the monomial and fundamental bases of quasisymmetric functions.

\subsection{Definition, and basic properties}
Let $\{x_i\}_{i\in\PP}$ be indeterminates, and let $\xx=(x_1,x_2,x_3,\ldots)$.
For a digraph $D=(V,A)$, we define
$$B_D(\xx;y,z)=\sum_{f:V\to\PP}\l(\prod_{v\in V}x_{f(v)}\r)y^{|f_A^{\As}|}z^{|f_A^{\Ds}|}.$$
We call this invariant the \emph{quasisymmetric $B$-polynomial}. 
It is clear that for any positive integer~$q$, 
\begin{equation}\label{eq:specialization-quasi}
B_D(1^q;y,z)~=~\sum_{f:V\to[q]}y^{|f_A^{\As}|}z^{|f_A^{\Ds}|}~=~B_D(q,y,z),
\end{equation}
where $1^q$ denotes the infinite sequence $(1,1,\ldots,1,0,0,\ldots)$, where the first $q$ coordinates are 1 and the remaining ones are 0. Hence $B_D(\xx;y,z)$ determines $B_D(q,y,z)$.

We now recall several specializations of $B_D(\xx;y,z)$ appearing in the literature.
Recall that for a graph $G=(V,E)$, the \emph{chromatic symmetric function} defined by Stanley in~\cite{Stanley:symmetric-chromatic} is
$$X_G(\xx)=\sum_{\substack{f:V\to\PP\\ f_E^{=}=\emptyset}}~\prod_{v\in V}x_{f(v)},$$
and the \emph{Tutte symmetric function} defined by Stanley in~\cite{Stanley:symmetric-chromatic-more} is 
$$S_G(\xx;y)=\sum_{f:V\to\PP}\l(\prod_{v\in V}x_{f(v)}\r)(1+y)^{|f^{=}_E|}.$$
The Tutte symmetric function is also equivalent to the $U$-polynomial defined by Noble and Welsh in~\cite{Noble-Welsh:Upolynomial}.
Recall also that in~\cite{Shareshian-Wachs:chromatic-quasisym-functions} Shareshian and Wachs defined the \emph{chromatic quasisymmetric function} of a digraph\footnote{Shareshian and Wachs actually consider labeled graphs instead of digraphs in~\cite{Shareshian-Wachs:chromatic-quasisym-functions}. Our equivalent definition~\eqref{eq:quasisym-chrom} is given in terms of the (acyclic) digraph $D$ obtained from the labeled graph by orienting every edge from the endpoint of smaller label to the endpoint of larger label.} $D=(V,A)$ as 
\begin{equation}\label{eq:quasisym-chrom}
\mX_D(\xx;y)=\sum_{\substack{f:V\to\PP \\ f_A^{=}=\emptyset}}\l(\prod_{v\in V}x_{f(v)}\r)y^{f_A^{\As}}.
\end{equation}
Clearly $S_G(\xx;-1)=X_G(\xx)$ for any graph $G$, and $\mX_D(\xx;1)=X_{\uD}(\xx)$ for any digraph $D$. 
We now explain the relations between these invariants and the quasisymmetric $B$-polynomial. 
First, it follows directly from the definitions that
\begin{equation}\label{eq:Share-inB} 
\mX_D(\xx;y)=[z^{|A|}]B_D(\xx;yz,z).
\end{equation}
Moreover, the relations between the $B$-polynomial and the Potts polynomial given in Theorem~\ref{thm:Bpoly-Tutte}
can easily be lifted to relations between $B_D(\xx;y,z)$ and $S_G(\xx,y)$:
\begin{thm}\label{thm:Potts123-multi}
For any graph $G=(V,E)$,
\begin{eqnarray}
B_{\orient{G}}(\xx;y,z)&=&(yz)^{|E|}S_G\l(\xx;\frac{1}{yz}-1\r),\label{eq:PottsOne-multi}\\
\frac{1}{2^{|E|}}\sum_{\vec{G}\in \ori(G)}B_{\vec{G}}(\xx;y,z)&=&\l(\frac{y+z}{2}\r)^{|E|}S_G\l(\xx;\frac{2}{y+z}-1\r).\label{eq:PottsTwo-multi}
\end{eqnarray}
Moreover, for any digraph $D$,
\begin{equation}
B_{D}(\xx;y,y)=y^{|E|}S_{\uD}\l(\xx;\frac{1}{y}-1\r).\label{eq:PottsThree-multi}
\end{equation}
\end{thm}
\begin{proof}
The proofs~\eqref{eq:PottsOne},~\eqref{eq:PottsTwo} and~\eqref{eq:PottsThree} extend verbatim to prove~\eqref{eq:PottsOne-multi},~\eqref{eq:PottsTwo-multi} and~\eqref{eq:PottsThree-multi}.
\end{proof}
Equations~\eqref{eq:Share-inB} and~\eqref{eq:PottsOne-multi} show that the quasisymmetric $B$-polynomial is a generalization of both the Tutte symmetric function (defined for graphs) and the chromatic quasisymmetric function (defined for acyclic digraphs).


In order to state other properties of $B_{D}(\xx;y,z)$, we need to recall some basic results about quasisymmetric functions.
Let $R$ be a ring. A \emph{quasisymmetric function in $\xx$} with coefficients in $R$ is a formal power series $f$ in the variables $\{x_i\}_{i\in \PP}$ with coefficients in $R$, such that the degrees of the monomials occurring in $f$ are bounded, and for all positive integers $k$, $\de_1,\ldots,\de_k$, $i_1<i_2<\ldots<i_k$ and $j_1<j_2<\ldots<j_k$,
\begin{equation}\label{eq:defQSym}
[x_{i_1}^{\de_1}x_{i_2}^{\de_2}\cdots x_{i_k}^{\de_k}]f=[x_{j_1}^{\de_1}x_{j_2}^{\de_2}\cdots x_{j_k}^{\de_k}]f.
\end{equation}
A \emph{symmetric function} is a quasisymmetric function such that~\eqref{eq:defQSym} holds for any tuples $(j_1,j_2,\ldots,j_k)$ of distinct integers (not necessarily increasing). 
We denote by $\QSym_R(\xx)$ (resp. $\Sym_R(\xx)$) the set of quasisymmetric functions (resp. symmetric function) in $\xx$ with coefficients in $R$. This set clearly has the structure of an $R$-algebra.
We denote by $\QSym_R^n(\xx)$ (resp. $\Sym_R^n(\xx)$) the submodule of $\QSym_R(\xx)$ (resp. $\Sym_R(\xx)$) made of the series $f$ which are homogeneous of degree $n$. Recall that a \emph{composition of $n$} is a tuple of positive integers summing to $n$, and that the notation $(\de_1,\ldots,\de_k)\vDash n$ means that $(\de_1,\ldots,\de_k)$ is a composition of $n$.
For $\de=(\de_1,\ldots,\de_k)\vDash n$, we define the \emph{quasisymmetric monomial function} $M_{\de}\in \QSym_R^n(\xx)$ by
$$M_\de=\sum_{i_1<i_2<\cdots<i_k}x_{i_1}^{\de_1}x_{i_2}^{\de_2}\cdots x_{i_k}^{\de_k},$$
where the sum is over increasing $k$-tuples of positive integers. 
It is clear that $\{M_\de\}_{\de\vDash n}$ is a basis of $\QSym_R^n(\xx)$. For $f\in\QSym_R(\xx)$, we denote by $[M_\de]f$ the coefficient of $M_\de$ in the expansion of $f$ in the basis $\{M_\de\}_{\de}$. 

It is intuitively clear that $B_D(\xx;y,z)$ is in $\QSym_{\ZZ[x,y]}^{|V|}(\xx)$ (since an order-preserving modification of the colors of a coloring $f$ of $D$ does not change the number of ascents and descents). We now give the expansion of $B_D(\xx;y,z)$ in the basis $\{M_\de\}_{\de\vDash |V|}$.
\begin{prop}\label{prop:monomial-basis}
For any digraph $D=(V,A)$, 
\begin{equation}\label{eq:monomial-basis}
B_D(\xx;y,z)=\sum_{p=1}^{|V|}\sum_{g\in \Surj(V,p)}M_{\l(|g^{-1}(1)|,|g^{-1}(2)|,\ldots,|g^{-1}(p)|\r)}y^{|g_A^{\As}|}z^{|g_A^{\Ds}|},
\end{equation}
where $\Surj(V,p)$ is the set of surjective maps from $V$ to $[p]$.
\end{prop}

\begin{example}\label{exp:Monomial}
For the digraphs $D$, $D'$, $D''$ represented in Figure~\ref{fig:exp-quasi}, one gets 
\begin{eqnarray*}
B_D(\xx;y,z)&=&(y+z)M_{(1,1)}+M_{(2)},\\
B_{D'}(\xx;y,z)&=&(y^2+z^2+4yz)M_{(1,1,1)}+(yz+y+z)(M_{(1,2)}+M_{(2,1)})+M_{(3)},\\
B_{D''}(\xx;y,z)&=&2(y^2+z^2+yz)M_{(1,1,1)}+(z^2+2y)M_{(1,2)}+(y^2+2z)M_{(2,1)}+M_{(3)}.
\end{eqnarray*}
\end{example}

\fig{width=.9\linewidth}{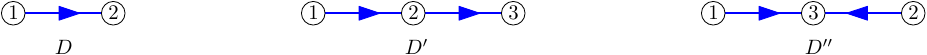}{Three (compatibly labeled) acyclic digraphs.}

\begin{proof}[Proof of Proposition~\ref{prop:monomial-basis}]
The proof is almost the same as that of Theorem~\ref{thm:existence-B}. Using the notation introduced in that proof, one gets
\begin{eqnarray*}
\sum_{f: V\to \PP}\l(\prod_{v\in V}x_{f(v)}\r)y^{|f_A^{\As}|}z^{|f_A^{\Ds}|}&=&\sum_{p=1}^{|V|}~\sum_{g\in \Surj(V,p)}y^{|g_A^{\As}|}z^{|g_A^{\Ds}|}\sum_{f:[V]\to [q]\textrm{ such that } \tf=g}\l(\prod_{v\in V}x_{f(v)}\r).
\end{eqnarray*}
Moreover, it is easy to see that for any $g\in \Surj(V,p)$, 
$$\ds\sum_{f:[V]\to [q]\textrm{ such that } \tf=g}\l(\prod_{v\in V}x_{f(v)}\r)=M_{\l(|g^{-1}(1)|,|g^{-1}(2)|,\ldots,|g^{-1}(p)|\r)},$$
which gives~\eqref{prop:monomial-basis}.
\end{proof}

\begin{remark}\label{rk:varphi}
Consider the linear map $\varphi$ from $\QSym_R(\xx)$ to $R[q]$ (the ring of polynomials in $q$ with coefficients in $R$) which for any composition $(\de_1,\ldots,\de_k)\vDash n$ sends the basis element $M_\de$ to the polynomial $\frac{q(q-1)\ldots(q-k+1)}{k!}$. It is clear that for any positive integer~$q$ the map $\varphi$ coincides with the evaluation map at $\xx=1^q$. Hence $\varphi\l(B_D(\xx;y,z)\r)=B_D(q,y,z)$, and Proposition~\ref{prop:monomial-basis} is a refinement of the expansion~\eqref{eq:q-falling-factorial} of $B_D(q,y,z)$ in the ``falling factorial'' basis of polynomials.
\end{remark}

We now state the analogue of Proposition~\ref{prop:easy} for $B_D(\xx;y,z)$. Let $\rho$ be the linear map on $\QSym_R(\xx)$ sending each basis element $M_{\de_1,\de_2,\ldots,\de_k}$ to $M_{\de_k,\ldots,\de_2,\de_1}$. It is clear that $\rho$ is an involution (intuitively, $\rho$ ``reverses the order of the variables'').

\begin{prop}\label{prop:easy-quasisym} 
For any digraph $D=(V,A)$, the quasisymmetric $B$-polynomial has the following properties.
\begin{compactenum}[(a')]
\item\label{it1quasi} $\rho(B_D(\xx;y,z))=B_D(\xx;z,y)$. 
\item\label{it2quasi} $B_{D^{-A}}(\xx;y,z)=B_D(\xx;z,y)$.
\item\label{it3quasi} $\ds \sum_{f: V\to \PP}\l(\prod_{v\in V}x_{f(v)}\r)u^{\l|f_A^=\r|}y^{\l|f_A^{\As}\r|}z^{\l|f_A^{\Ds}\r|}=u^{|A|}B_D\l(\xx;\frac{y}{u},\frac{z}{u}\r)$.
\item\label{it4quasi} $\ds \sum_{f: V\to \PP}\l(\prod_{v\in V}x_{f(v)}\r)y^{\l|f_A^{\wAs}\r|}z^{\l|f_A^{\wDs}\r|}=(yz)^{|A|}B_D\l(\xx;\frac{1}{z},\frac{1}{y}\r)$.
\item\label{it5quasi} The quasisymmetric $B$-polynomial of the digraph with a single vertex and no arcs is $M_{(1)}=\sum_{i\in \PP}x_i$.
\item\label{it6quasi} If $a=(u,u)\in A$ is a loop, then $B_{D_{\setminus a}}(\xx;y,z)=B_D(\xx;y,z)$.
\item\label{it7quasi} If $D$ is the disjoint union of two digraphs $D_1$ and $D_2$, then $B_{D}(\xx;y,z)=B_{D_1}(\xx;y,z)\,B_{D_2}(\xx;y,z)$.
\item\label{it8quasi} 
If $D$ has no loops, $\deg_y(B_D(\xx;y,y))=|A|$.
\addtocounter{enumi}{1}
\item\label{it10quasi} The expansion of the polynomial $B_D(\xx;y,z)$ in the basis $\{M_{\de}y^iz^j\}_{\de\vDash |V|,i,j\geq 0}$ has positive integer coefficients. 
\item\label{it11quasi} $[M_{1^{|V|}}]B_D(\xx;1,0)=\sum_{\textrm{bijection }f:V\to[|V|]}y^{\l|f_A^{\As}\r|}z^{\l|f_A^{\Ds}\r|}$.
\item\label{it12quasi} For any $k$ in $[|V|]$, $[M_{(k,|V|-k)}]B_D(\xx;1,0)$ is the number of subsets of vertices $U$ of size $k$ such that every arc joining $U$ to $V\setminus U$ is oriented away from $U$.
\end{compactenum}
\end{prop}

\begin{proof} 
Property~\eqref{it1quasi} follows from~\eqref{eq:monomial-basis} by considering the involution on $\Surj(V,p)$ which associates to any function $f\in\Surj(V,p)$ the function $\bar{f}:v\mapsto p+1-f(v)$. The other assertions are proved with the same arguments as for Proposition~\ref{prop:easy}.
\end{proof}

\begin{remark}\label{rk:detect-quasi}
Observe that 
$$[M_{1,|V|-1}]B_D(\xx;y,z)=\sum_{v\in V}y^{\out(v)}z^{\ind(v)},$$
where $\out(v)$ and $\ind(v)$ are the outdegree and indegree of $v$ respectively. Hence, the indegree and outdegree distribution can be read off $B_D(\xx;x,y)$. In particular, knowing $B_D(\xx;x,y)$ (and $|A|$) is sufficient to detect whether $D$ is a \emph{directed tree} (an oriented tree such that every vertex except one has indegree 1). This is in contrast with $B_D(q;y,z)$ as explained in Remark~\ref{rk:not-detected}.

We now show that the \emph{profile} of an acyclic digraph can be obtained from its quasisymmetric $B$-polynomial. For a vertex $v$ of an acyclic digraph $D$, we call \emph{height} the maximal number of vertices on a directed path of $D$ ending at $v$. We call \emph{profile} of $D$, the composition $\de=(\de_1,\ldots,\de_k)\vDash |V|$, where $\de_i$ is the number of vertices having height $i$. It is not hard to see that the profile of a digraph $D=(V,A)$ is the largest composition $\de$ for the lexicographic order, such that $[M_{\de}][y^{|A|}]B_D(\xx;y,z)\neq 0$. 
\end{remark}

Next, we define the quasisymmetric version of the strict and weak chromatic polynomials.
We define
$$\chi^{\As}_D(\xx)=[y^A]B_D(\xx;y,1)=\sum_{f:V\to\PP,~ f^{\As}_A=A}\prod_{v\in V}x_{f(v)},$$
and 
$$\chi^{\wAs}_D(\xx)=B_D(\xx;1,0)=\sum_{f:V\to\PP,~ f^{\wAs}_A=A}\prod_{v\in V}x_{f(v)}.$$
 
The first two identities of Theorem~\ref{thm:expansions} can be lifted to the quasisymmetric setting.
\begin{thm}
For any digraph $D=(V,A)$,
\begin{eqnarray}
\sum_{R\uplus S\uplus T= A} y^{|S|}z^{|T|}\chrom_{D^{-T}_{\setminus R}}^{\As}(\xx)&=& \Q_D(\xx;1+y,1+z),\label{>Delete-quasi}\\
\sum_{R\uplus S\uplus T= A} y^{|S|}z^{|T|}\chrom_{D^{-T}_{\setminus R}}^{\wAs}(\xx)&=&(1+y+z)^{|A|}\B_D\left(\xx;\frac{1+y}{1+y+z},\frac{1+z}{1+y+z}\right).\label{geqDelete-quasi}
\end{eqnarray}
\end{thm}
\begin{proof} The proof is the same as for~\eqref{>Delete} and~\eqref{geqDelete}. 
\end{proof}

\subsection{Expansion in the fundamental basis, via the theory of $P$-partitions}
We will now use the theory of $P$-partitions in order to express $\chi^{\As}_D(\xx)$, $\chi^{\wAs}_D(\xx)$ and $B_D(\xx,y,1)$ in terms of fundamental quasisymmetric functions.
For a subset $S$ of $[n-1]$, we define the \emph{fundamental quasisymmetric function} $F_{n,S}\in \QSym_R^n(\xx)$ by 
$$F_{n,S}=\sum_{\substack{i_1\leq i_2\leq \cdots\leq i_n\\ i_s<i_{s+1} \textrm{ if }s\in S}}x_{i_1}x_{i_2}\cdots x_{i_n}.$$
It is well known that $\{F_{n,S}\}_{S\subseteq [n-1]}$ is a basis of $\QSym_R^n(\xx)$ (see e.g.~\cite{Stanley:volume2}). For $f\in\QSym_R^n(\xx)$, we denote by $[F_{n,S}]f$ the coefficient of $F_{n,S}$ in the expansion of $f$ in the basis $\{F_{n,S}\}_{S\subseteq [n-1]}$. 

\begin{remark}\label{rk:varphi2}
Consider the linear map $\varphi$ from $\QSym_R(\xx)$ to $R[q]$ sending each basis element $F_{n,S}$ to the polynomial $\frac{(q-|S|)(q-|S|+1)\ldots(q-|S|+n-1)}{n!}$. It is clear that for any positive integer~$q$, the map $\varphi$ coincides with the evaluation map at $\xx=1^q$. Hence, $\varphi$ is the same as the linear map considered in Remark~\ref{rk:varphi}, and $\varphi\l(B_D(\xx;y,z)\r)=B_D(q,y,z)$.
\end{remark}

We now recall some basic results about $P$-partitions. The reader can refer to~\cite{Stanley:volume1-ed2} for some background. For convenience, we will state definitions and results in terms of digraphs instead of posets.


A digraph with vertex set $[n]$ is called a \emph{labeled digraph}. We say that a labeled digraph $D$ has a \emph{compatible labeling} (resp. \emph{anticompatible labeling}) if for all $(u,v)\in A$, $u<v$ (resp. $u>v$).
\begin{definition}
Let $D=([n],A)$ be a labeled directed graph.
A \emph{$D$-partition} is a function $f:[n]\to \PP$ such that 
\begin{compactitem}
\item for all $(u,v)\in A$, $f(u)\leq f(v)$, 
\item for all $(u,v)\in A$ with $u<v$, $f(u)<f(v)$.
\end{compactitem}
We denote by $\mP_D$ the set of $D$-partitions\footnote{The reader familiar with the theory of $P$-partitions will recognize that $\mP_D$ is the set of $(P,\Id)$-partitions for the poset $P=([n],\preceq)$, where $u\preceq v$ means that there exists a directed path from $v$ to $u$.}, and by $\mP_D(q)$ the subset of $D$-partitions $f$ such that $f(V)\subseteq [q]$.
\end{definition}

\begin{example}
For the labeled digraph $D$ on the left of Figure~\ref{fig:P-partition}, we get $\mP_{D}=\{f:[3]\to \PP~|~f(1)<f(2) \textrm{ and } f(1)<f(3)\}$.
For the labeled digraph $D'$ on the right of Figure~\ref{fig:P-partition}, we get $\mP_{D'}=\{f:[3]\to \PP~|~f(2)\leq f(1) \textrm{ and } f(2)<f(3)\}$.
\end{example}

\fig{width=.4\linewidth}{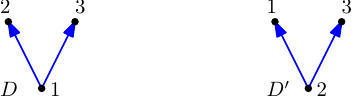}{Two labeled digraphs.}

Clearly, $\mP_D$ is empty unless $D$ is acyclic. 
Moreover, if $D$ is an acyclic digraph with a compatible labeling, then $\mP_D(q)$ is the set of $q$-colorings with only strict ascents, so that $|\mP_D(q)|=\chi^{\As}_D(q)$. 
Similarly, if $D$ is an acyclic digraph with an anticompatible labeling, $\mP_{D}(q)$ is the set of $q$-colorings with only weak ascents, so that $|\mP_D(q)|=\chi^{\wAs}_D(q)$.


We call \emph{linear extension} of a digraph $D=([n],A)$ a permutation $\si\in\fS_n$ such that for all $(u,v)\in A$, $\si^{-1}(u)< \si^{-1}(v)$ (equivalently, $u$ appears before $v$ in the \emph{one-line notation} $\si(1)\si(2)\cdots\si(n)$ of $\si$).
Let $\mL(D)$ be the set of linear extensions of $D$. For instance, for the digraph $D$ in Figure~\ref{fig:P-partition}, $\mL(D)=\{123,132\}$. 
For a permutation $\si$ of $[n]$, we denote by $\mP_{\si}$ the set of functions $f:[n]\to \PP$ such that 
\begin{compactitem}
\item for all $i\in[n-1]$, $f(\si(i))\leq f(\si(i+1))$,
\item for all $i\in[n-1]$ with $\si(i)<\si(i+1)$, $f(\si(i))<f(\si(i+1))$.
\end{compactitem}

We now state, in terms of labeled digraphs, the fundamental lemma of $P$-partitions (see~\cite[Lemma 3.15.3]{Stanley:volume1-ed2}).
\begin{lemma}\label{lem:Ppartition-lemma}
For any labeled digraph $D=([n],A)$,
$$\mP_{D}=\biguplus_{\si \in \mL(D)}\mP_\si.$$
\end{lemma}

\begin{example}
For the labeled digraph $D$ on the left of Figure~\ref{fig:P-partition}, we have $\mL(D)=\{\si,\pi\}$ with $\si=123$ and $\pi=132$.
Accordingly, 
$$\mP_{D}=\mP_\si\uplus\mP_\pi=\{f:[3]\to \PP~|~f(1)<f(2)<f(3)\}\uplus\{f:[3]\to \PP~|~f(1)<f(3)\leq f(2)\}.$$
For the labeled digraph $D'$ on the right of Figure~\ref{fig:P-partition}, we have $\mL(D')=\{\si',\pi'\}$ with $\si'=213$ and $\pi'=231$.
Accordingly, 
$$\mP_{D'}=\mP_{\si'}\uplus\mP_{\pi'}=\{f:[3]\to \PP~|~f(2)\leq f(1)<f(3)\}\uplus\{f:[3]\to \PP~|~f(2)<f(3)\leq f(1)\}.$$
\end{example}


The fundamental lemma~\ref{lem:Ppartition-lemma} implies the following result about the invariants $\chi^{\As}_D(\xx)$ and $\chi^{\wAs}_D(\xx)$, which is simply a reformulation of the classical result~\cite[Corollary 7.19.5]{Stanley:volume2} in our digraph setting.

\begin{lemma}\label{lem:chi-in-Fn}
For any compatibly labeled acyclic digraph $D=([n],A)$,
$$\chi^{\As}_D(\xx)=\sum_{\si\in \mL(D)}F_{n,\Asc(\si)},$$
where $\Asc(\si)=\{i\in [n-1]~|~\si(i)<\si(i+1)\}$ is the ascent set of $\si$.
For any anticompatibly labeled acyclic digraph $D=([n],A)$,
$$\chi^{\wAs}_D(\xx)=\sum_{\si\in \mL(D)}F_{n,\Asc(\si)}.$$
\end{lemma}

\begin{proof}
By definition, for any permutation $\si \in \fS_n$, 
$$\sum_{f\in\mP_\si}\prod_{v\in V}x_{f(v)}~=~\sum_{\substack{f(\si(1))\leq f(\si(2))\leq\cdots\leq f(\si(n))\\ \forall i\in \Asc(\si),~f(\si(i))<f(\si(i+1))}}x_{f(\si(1))}x_{f(\si(2))}\cdots x_{f(\si(n))}~=~F_{n,\Asc(\si)}.$$
Hence, by Lemma~\ref{lem:Ppartition-lemma}, for any labeled digraph $D$,
$$\sum_{f\in\mP_D}\prod_{v\in V}x_{f(v)}=\sum_{\si\in\mL(D)}\sum_{f\in\mP_\si}\prod_{v\in V}x_{f(v)}=\sum_{\si\in\mL(D)}F_{n,\Asc(\si)}.$$
Moreover, it is clear from the definitions, that for any compatibly labeled digraph $D$, 
$$\chi^{\As}_D(\xx)=\sum_{f\in\mP_D}\prod_{v\in V}x_{f(v)},$$
because $\mP_D=\{f:V\to\PP,~ f^{\As}_A=A\}$.
Similarly, for any anticompatibly labeled digraph $D$, 
$$\sum_{f\in\mP_D}\prod_{v\in V}x_{f(v)}=\chi^{\wAs}_D(\xx),$$
because $\mP_D=\{f:V\to\PP,~ f^{\wAs}_A=A\}$.
\end{proof}

From Lemma~\ref{lem:chi-in-Fn}, we obtain an expression for $B_D(\xx;y,1)$ when $D$ is acyclic.

\begin{thm}\label{thm:B-fundamental-basis}
If $D=([n],A)$ is a compatibly labeled acyclic digraph, then 
\begin{equation}\label{eq:B-fundamental-basis}
B_D(\xx;y,1)=\sum_{\si\in\fS_n}F_{n,\Asc(\si^{-1})}y^{|\si_A^{\As}|},
\end{equation}
where $\Asc(\si^{-1})=\{i\in [n-1]~|~\si^{-1}(i)<\si^{-1}(i+1)\}$, and $\si_A^{\As}=\{(u,v)\in A,\si(v)>\si(u)\}$. In particular (by Remark~\ref{rk:varphi2}),
$$B_D(q,y,1)=\frac{1}{n!} \sum_{\si\in\fS_n}\l(\prod_{i=1}^{n}\l(q-\l|\Asc(\si^{-1})\r|+i-1\r)\r)y^{|\si_A^{\As}|}.$$
\end{thm}

\begin{proof}
Setting $z=0$ in~\eqref{>Delete-quasi} gives
$$B_D(\xx;y+1,1)=\sum_{R\subseteq A} y^{A\setminus R}\chrom_{D_{\setminus R}}^{\As}(\xx).$$
Since for any set $R$ the digraph $D_{\setminus R}$ is compatibly labeled, Lemma~\ref{lem:chi-in-Fn} gives
$$B_D(\xx;y+1,1)=\sum_{R\subseteq A} y^{|A\setminus R|}\sum_{\si\in \mL(D_{\del R})}F_{n,\Asc(\si)}=\sum_{\si\in \fS_n} F_{n,\Asc(\si)} \sum_{R\subseteq A,\textrm{ such that } \si\in\mL(D_{\del R})}y^{|A\setminus R|}.$$
Now, observe that $\si$ is in $\mL(D_{\del R})$ if and only if $A\setminus R\subseteq {\si^{-1}}^\As_A$. Thus, 
$$B_D(\xx;y+1,1)=\sum_{\si\in \fS_n} F_{n,\Asc(\si)} \sum_{S\subseteq {\si^{-1}}^\As_A}y^{|S|}=\sum_{\si\in \fS_n} F_{n,\Asc(\si)}(y+1)^{|{\si^{-1}}_A^{\As}|},$$
which is equivalent to~\eqref{eq:B-fundamental-basis}.
\end{proof}

\begin{example}
Consider the digraphs $D'$ and $D''$ of Figure~\ref{fig:exp-quasi}. The expansion of $B_{D'}(\xx;y,z)$ and $B_{D''}(\xx;y,z)$ in the basis $\{M_\de\}_{\de}$ is given in Example~\ref{exp:Monomial}. 
Using the change of basis\footnote{The change of basis is given by $\ds M_\de=\sum_{S(\de)\subseteq R\subseteq [n-1]}(-1)^{|R\setminus S(\de)|}F_{n,R}$, where $\ds S(\de_1,\ldots,\de_k)=\l\{\sum_{i=1}^j\de_i,~j\in[k-1]\r\}$.} between $\{M_\de\}_{\de\vDash n}$ and $\{F_{n,S}\}_{S\subseteq [n-1]}$ gives
\begin{eqnarray}
B_{D'}(\xx;y,z)&\!\!=\!\!&(y^2+z^2+2yz-2y-2z+1)F_{3,\{1,2\}}+(yz+y+z-1)(F_{3,\{1\}}+F_{3,\{2\}})+F_{3,\emptyset},\label{eq:BD'}\\
B_{D''}(\xx;y,z)&\!\!=\!\!&(y^2+z^2+2yz-2y-2z+1)F_{3,\{1,2\}}+(z^2+2y-1)F_{3,\{1\}}+(y^2+2z-1)F_{3,\{2\}}+F_{3,\emptyset}.\quad \quad \label{eq:BD''}
\end{eqnarray}
The expression for $B_{D'}(\xx,y,1)$ and $B_{D''}(\xx,y,1)$ can be checked to match those given by Theorem~\ref{thm:B-fundamental-basis} as computed in Tables~\ref{tab:exp-fundaD'} and~\ref{tab:exp-fundaD''}.
\end{example}

\begin{table}[h!]
\centering 
 \begin{tabular}{c|c c c c c c} 
$\si$ & $123$ & $132$ & $213$ & $231$ & $312$ & $321$ \\
$\si^{-1}$ & $123$ & $132$ & $213$ & $312$ & $231$ & $321$ \\ \hline
$|\si_{A}^{\As}|$& 2 & 1 & 1 & 1 & 1 & 0\\ 
$\Asc(\si^{-1})$& $\{1,2\}$ & $\{1\}$ & $\{2\}$ & $\{2\}$ & $\{1\}$ & $\emptyset$\\
$\Asc_\prec(\si^{-1})$& $\emptyset$ & $\{1\}$ & $\{2\}$ & $\emptyset$ & $\emptyset$ & $\emptyset$\\\hline 
Contribution to $B_{D'}(\xx;y,1)$& $y^2F_{3,\{1,2\}}$ & $yF_{3,\{1\}}$ & $yF_{3,\{2\}}$ & $yF_{3,\{2\}}$ & $yF_{3,\{1\}}$ & $F_{3,\emptyset}$\\ 
Contribution to $\om\l([z^{|A|}]B_{D'}(\xx;yz,z)\r)$& $y^2F_{3,\emptyset}$ &$yF_{3,\{1\}}$ &$yF_{3,\{2\}}$ &$yF_{3,\emptyset}$ &$yF_{3,\emptyset}$& $F_{3,\emptyset}$\\[2mm] 
\end{tabular}
\caption{Illustration of Formulas~\eqref{eq:B-fundamental-basis} and~\eqref{eq:Shareshian-Wachs} for the labeled digraph $D'$ represented in Figure~\ref{fig:exp-quasi} (center). In the computation of $\Des_{\prec}$, one uses the order $\prec$ for which the unique relation is $1\prec 3$.}\label{tab:exp-fundaD'}
\end{table}

\begin{table}[h!]
\centering 
 \begin{tabular}{c|c c c c c c} 
$\si$ & $123$ & $132$ & $213$ & $231$ & $312$ & $321$ \\
$\si^{-1}$ & $123$ & $132$ & $213$ & $312$ & $231$ & $321$ \\ \hline
$|\si_{A}^{\As}|$& 2 & 1 & 2 & 0 & 1 & 0\\ 
$\Asc(\si^{-1})$& $\{1,2\}$ & $\{1\}$ & $\{2\}$ & $\{2\}$ & $\{1\}$ & $\emptyset$\\
$\Asc_\prec(\si^{-1})$& $\{1\}$ & $\emptyset$ & $\emptyset$ & $\{2\}$ & $\emptyset$ & $\emptyset$\\\hline 
Contribution to $B_{D'}(\xx;y,1)$& $y^2F_{3,\{1,2\}}$ & $yF_{3,\{1\}}$ & $y^2F_{3,\{2\}}$ & $F_{3,\{2\}}$ & $yF_{3,\{1\}}$ & $F_{3,\emptyset}$\\ 
Contribution to $\om\l([z^{|A|}]B_{D'}(\xx;yz,z)\r)$& $y^2F_{3,\{1\}}$ &$yF_{3,\emptyset}$ &$y^2F_{3,\emptyset}$ &$F_{3,\{2\}}$ &$yF_{3,\emptyset}$& $F_{3,\emptyset}$\\[2mm] 
\end{tabular}
\caption{Illustration of Formulas~\eqref{eq:B-fundamental-basis} and~\eqref{eq:Shareshian-Wachs} for the labeled digraph $D''$ represented in Figure~\ref{fig:exp-quasi} (right). In the computation of $\Des_{\prec}$, one uses the order $\prec$ for which the unique relation is $1\prec 2$.}\label{tab:exp-fundaD''}
\end{table}

\begin{example}\label{exp:exp-permut-stats-quasi}
Let us illustrate Theorem~\ref{thm:B-fundamental-basis} for the digraphs $D$, $D'$, $D''$ represented in Figure~\ref{fig:exp-permut-stats}. For the directed path on $n$ vertices, $D=([n],\{(i,i+1)~|~i\in[n-1]\})$, Theorem~\ref{thm:B-fundamental-basis} gives
\begin{equation}\label{eq:exp-path-quasi}
B_{D}(\xx;y,1)=\sum_{\si\in \fS_n}F_{n,\Asc(\si^{-1})}y^{|\Asc(\si)|}.
\end{equation}
For $D'=([n],A)$ the digraph with $i$ copies of the arc $(i,i+1)$ for all $i\in[n-1]$, Theorem~\ref{thm:B-fundamental-basis} gives
\begin{equation}\label{eq:exp-thickening-path-quasi}
B_{D}(\xx;y,1)=\sum_{\si\in \fS_n}F_{n,\Asc(\si^{-1})}y^{{n \choose 2}-\maj(\si)},
\end{equation}
where $\ds \maj(\si)=\sum_{\substack{i\in[n-1]\\ \si(i)>\si(i+1)}}\!\!\!\!\!\!\!i~$ is the \emph{major index}. Lastly, for $D''=([n],\{(u,v) ~|~ 1\leq u<v\leq n\})$ Theorem~\ref{thm:B-fundamental-basis} gives
\begin{equation}\label{eq:exp-Kn-quasi}
B_{D}(\xx;y,1)=\sum_{\si\in \fS_n}F_{n,\Asc(\si^{-1})}y^{{n \choose 2}-\inv(\si)},
\end{equation}
where $\inv(\si)$ is the number of pairs $(i,j)\in[n]^2$ with $i<j$ and $\si(i)>\si(j)$. 
\end{example}

\subsection{Quasisymmetric function duality, and symmetries of the $B$-polynomial}
In this subsection, we discuss certain results related to the \emph{duality map} $\om$ on quasisymmetric functions. Let $\om$ be the linear map on $\QSym_R(\xx)$ sending each basis element $F_{n,S}$ to $F_{n,[n-1]\setminus S}$. It is clear that $\om$ is an involution. It is also well known that $\om$ is a ring homomorphism (in other words, $w(fg)=w(f)w(g)$ for all $f,g\in\QSym_R(\xx)$). Consequently, the restriction of $\om$ to $\Sym_R(\xx)$ is the well-known \emph{duality map}, which is the unique ring homomorphism sending $e_n=F_{n,[n-1]}$ to $h_n=F_{n,\emptyset}$ for all $n$. We now state a key lemma which is an easy consequence of a well-known fact about the duality for quasisymmetric functions associated to labeled posets. 

\begin{lem}\label{lem:duality-quasisym}
For any digraph $D$, $\ds ~\om\l(\chi^{\As}_D(\xx)\r)~=~\ONE_{D \textrm{ acyclic}}\cdot \chi^{\wAs}_D(\xx)$.
\end{lem}

\begin{proof}
If $D$ is not acyclic, then clearly $\chi^{\As}_D(\xx)=0$ (because a coloring cannot be strictly increasing along a directed cycle). 
Suppose now that $D$ is acyclic. Up to renaming the vertices of $D$, we can assume that $D=([n],A)$ is compatibly labeled. Now let $\wD=([n],\wA)$ be the (anticompatibly labeled) digraph obtained from $D$ by relabeling the vertex $i$ by $n+i-1$ for all $i\in [n]$. Let $\theta\in\fS_n$ be the permutation defined by $\theta(i)=n+1-i$ for all $i\in[n]$. It is clear that the mapping $\Phi$ defined on $\mL(D)$ by $\phi(\si)=\theta\circ \si$ is a bijection between $\mL(D)$ and $\mL(\wD)$. Moreover, $\Asc(\si)=[n-1]\setminus \Asc(\phi(\si))$, so that
$$\chi^{\As}_D(\xx)=\sum_{\si\in \mL(D)}F_{n,\Asc(\si)}=\sum_{\pi\in \mL(\wD)}F_{n,[n-1]\setminus \Asc(\pi)}=\om(\chi^{\wAs}_D(\xx)).$$
\end{proof}

We now state the main result of this subsection.

\begin{thm} \label{thm:duality-quasisym}
For any digraph $D$,
\begin{eqnarray}
\sum_{\substack{R\uplus S\uplus T= A \\ D^{-T}_{\setminus R}\textrm{ acyclic}}} y^{|S|}z^{|T|}\chrom_{D^{-T}_{\setminus R}}^{\wAs}(\xx)&=&\om\l( \B_D(\xx,1+y,1+z)\r).\label{eq:q-1-quasi}
\end{eqnarray}
\end{thm}

\begin{proof} Apply the involution $\om$ to both sides of~\eqref{>Delete-quasi} and use Lemma~\ref{lem:duality-quasisym}.
\end{proof}

\begin{remark} \label{rk:varphi-duality}
Recall the mapping $\varphi$ defined in Remark \ref{rk:varphi}, which is such that $\varphi(B_D(\xx,y,z))=B_D(q,y,z)$. It is well known (and easy to see from Remark \ref{rk:varphi2}), that for any $f\in\QSym_R^n(\xx)$ there is simple relation between $\varphi(f)$ and  $\varphi(\om(f))$. Namely, denoting $\varphi(f)=P(q)$, one has  $\varphi(\om(f))=(-1)^nP(-q)$. Thus, \eqref{eq:q-1-quasi} is a refinement of  \eqref{eq:q-1}.
\end{remark}

Theorem \ref{thm:duality-quasisym} implies some symmetries for $B_D(\xx,y,1)$ in the case of acyclic digraphs.

\begin{cor}\label{cor:symmetries-quasi}
If $D=(V,A)$ is an acyclic digraph, then
\begin{equation}\label{eq:symmetry1-quasi}
\om\l(B_D(\xx;y,1)\r)=\rho\l(y^{|A|}B_D(\xx;1/y,1)\r).
\end{equation}
If the graph underlying $D$ is a forest, then
\begin{equation}\label{eq:symmetry2-quasi}
\om\l(B_D(\xx;y,z)\r)=(y+z-1)^{|A|}B_D\l(\xx;\frac{y}{y+z-1},\frac{z}{y+z-1}\r).
\end{equation}
\end{cor}

%

\begin{proof}
If $D$ is acyclic, then all its subgraphs are acyclic. Hence, setting $z=0$ in~\eqref{eq:q-1-quasi} gives
$$\sum_{R\subseteq A} y^{|A\setminus R|} \chrom_{D_{\del R}}^{\wAs}(\xx)=\om\l(\B_D(\xx;y+1,1)\r).$$
Comparing this to the specialization $z=0$ of~\eqref{geqDelete-quasi} gives $\ds \om\l(\B_D(\xx,y+1,1)\r)=(y+1)^{|A|}\B_D(\xx;1,1/(y+1))$ which is equivalent to~\eqref{eq:symmetry1-quasi}.

If the graph underlying $D$ is a forest, then any digraph obtained from $D$ by deleting or reorienting arcs is acyclic. Hence,~\eqref{eq:q-1-quasi} gives
$$\sum_{R\uplus S\uplus T= A} y^{|S|}z^{|T|}\chrom_{D^{-T}_{\setminus R}}^{\wAs}(\xx)=\om\l(\B_D(\xx;1+y,1+z)\r).$$
Comparing this to~\eqref{geqDelete-quasi} gives $\ds \om\l(\B_D(\xx;1+y,1+z)\r)=(1+y+z)^{|A|}\B_D\left(\xx;\frac{1+y}{1+y+z},\frac{1+z}{1+y+z}\right)$, which is equivalent to~\eqref{eq:symmetry2-quasi}.
\end{proof}

\begin{remark}
Corrolary~\ref{cor:symmetries-quasi} gives the following result for $B_D(q,y,1)$ (which could have been proved directly in Section~\ref{sec:Ehrhart}): if $D=(V,A)$ is acyclic then
$\ds B_D(-q,y,1)=(-1)^{|V|}y^{|A|}B_D(q,1/y,1)$, and if $\uD$ is a forest then $\ds B_D(-q,y,z)=(-1)^{|V|}(y+z-1)^{|A|}B_D\l(q,\frac{y}{y+z-1},\frac{z}{y+z-1}\r)$. 
\end{remark}
 

\begin{remark}
In~\cite{Shareshian-Wachs:chromatic-quasisym-functions} a surprising identity is proved for $\mX_D(\xx;y):=[z^{|A|}]B_D(\xx;yz,z)$ for particular digraphs. We state this result with our notation for the reader's convenience. Let $D=([n],A)$ be a compatibly labeled acyclic digraph without double arcs. Suppose that there exists a partial order $\prec$ on $V$ such that there is an arc $(u,v)\in A$ if and only if $u,v$ are incomparable for $\prec$. Then,~\cite[Theorem 3.1]{Shareshian-Wachs:chromatic-quasisym-functions} states that
\begin{equation}\label{eq:Shareshian-Wachs}
\om([z^{|A|}]B_D(\xx;yz,z))=\sum_{\si\in\fS_n}F_{n,\Asc_{\prec}(\si^{-1})}y^{|\si^{\As}_A|},
\end{equation}
where $\Asc_{\prec}(\pi)=\{i\in[n-1]~|~\pi(i)\prec \pi(i+1)\}$.
For instance, for the digraphs $D'$ and $D''$ of Figure~\ref{fig:exp-quasi}, the expansion of $B_{D'}(\xx;y,z)$ and $B_{D''}(\xx;y,z)$ in the fundamental basis is given in~\eqref{eq:BD'} and~\eqref{eq:BD''}. 
The expression for $[z^{|A|}]B_{D'}(\xx,yz,z)$ and $[z^{|A|}]B_{D''}(\xx,yz,z)$ can be checked to match those given by~\eqref{eq:Shareshian-Wachs} as computed in Tables~\ref{tab:exp-fundaD'} and~\ref{tab:exp-fundaD''}.
\end{remark}


\smallskip
\section{A family of invariants generalizing the $B$-polynomial}\label{sec:B-family}
In this section we define a family of digraph invariants generalizing the $B$-polynomial. 

\begin{thm}\label{thm:def-Qm}
Let $m$ be a positive integer. For any digraph $D=(V,A)$ there exists a unique polynomial in $2m+1$ variables, denoted by $\Q_D^{(m)}(q;y_1,\ldots,y_m;z_1,\ldots,z_m)$, such that for every non-negative integer~$p$,
\begin{equation}\label{def:Qm}
\Q_D^{(m)}(mp+1;y_1,\ldots,y_m;z_1,\ldots,z_m)=\sum_{f:V\to [mp+1]}\prod_{k=1}^m y_k^{|f_A^k|}z_k^{|f_A^{-k}|},
\end{equation}
where $f_A^k$ (resp. $f_A^{-k}$) is the set of arcs $(u,v)\in A$ such that $f(v)-f(u)\in [(k-1)p+1\,..\,kp]$ (resp. $f(u)-f(v)\in [(k-1)p+1\,..\,kp]$). 
\end{thm}

For instance, $\Q_D^{(1)}(q;y_1;z_1)$ is equal to the $B$-polynomial $B_D(q,y_1,z_1)$. The proof of Theorem~\ref{thm:def-Qm} uses the following classical extension of Ehrhart's Theory to rational polytopes.

\begin{lemma}\label{lem:Ehrhart-quasi}
Let $\Pi\subset \RR^n$ be a polytope and let $m$ be an integer. Suppose that all the coordinates of the vertices of $\Pi$ are in $\frac{1}{m}\ZZ$. 
Then, there exists polynomials $E_0,E_1,\ldots,E_{m-1}$ such that for all non-negative integers $p$ and all $r\in[0..m-1]$,
$$\left| (mp+r)\Pi\cap \ZZ^n\right|=E_r(mp+r).$$
Moreover, for positive integer $p$ and all $r\in[0..m-1]$,
\begin{equation}\label{eq:Ehrhart-reciprocity-quasi}
\left|(mp-r)\Pi^\bu \cap \ZZ^n\right|=(-1)^{\dim(\Pi)}E_{r}(-mp+r),
\end{equation}
where $\Pi^\bu$ is the \emph{relative interior} of $\Pi$ (that is, the interior of $\Pi$ in the smallest affine subspace containing $\Pi$), and $\dim(\Pi)$ is the dimension of the smallest affine subspace containing $\Pi$.
\end{lemma}

\begin{example}
The 2-dimensional polytope $\Pi=[0,1/3]^2$ has vertex-coordinates in $\frac{1}{m}\ZZ$ for $m=3$. The polynomials are $E_i(q)=(\frac{q-i}{3}+1)^2$ for all $i\in\{0,1,2\}$. Since $\Pi^\bu=(0,1/3)^2$, we get $\left|\Pi^\bu \cap \ZZ^2\right|=0=E_2(-1)$, which is in accordance with~\eqref{eq:Ehrhart-reciprocity-quasi}.
\end{example}

\begin{proof}[Proof of Theorem~\ref{thm:def-Qm}]
The uniqueness part is obvious, so we focus on proving the existence of $\Q_D^{(m)}$. For a function $g:A\to[-m..m]$, we denote by $\chi^g_{D}(mp+1)$ the number of $(mp+1)$-colorings $f:V\to [mp+1]$ such that for all $k\in [m]$, $f_A^k=g^{-1}(k)$ and $f_A^{-k}=g^{-1}(-k)$.
Clearly, 
$$\sum_{f:V\to [mp+1]}\prod_{k=1}^m y_k^{|f_A^k|}z_k^{|f_A^{-k}|}=\sum_{g:A\to [-m..m]}\chi^g_{D}(mp+1)\prod_{k=1}^my_k^{\l|g^{-1}(k)\r|}z_k^{\l|g^{-1}(-k)\r|}.$$
Hence, it suffices to prove that for any function $g:A\to [-m..m]$, there exists a polynomial $P$ such that $\chi^g_{D}(mp+1)=P(mp+1)$ for every non-negative integer $p$.

To a function $g:A\to [-m..m]$, we associate a region $\De_D^g$ of $\RR^{V}$ defined at follows:
$$\De_D^g=\{(x_v)_{v\in V}\in \RR^{V}~|~\forall v\in V, x_v\in(0,1],\textrm{ and }\forall a=(u,v)\in A, x_v-x_u\in I_{g(a)}\},$$
where $I_0=\{0\}$ and for all $k\in [m]$, $I_{k}=((k-1)/m,k/m)$ and $I_{-k}=-I_k=(-k/m,-(k-1)/m)$.
Clearly, for any non-negative integer $p$, 
$$\chi_{D}^g(mp+1)=\l|(mp+1)\De_D^g \cap\ZZ^V \r|,$$
because for all $k\in[m]$, $(mp+1)I_k\cap \ZZ=[(k-1)p+1\,..\,kp]$.

For a subset of vertices $U\subseteq V$, we write 
$$\De_{D,U}^g=\{(x_v)_{v\in V}\in \RR^{V}~|~\forall v\in U, x_v=1, \forall v\in V\setminus U, x_v\in (0,1),\textrm{ and }\forall a=(u,v)\in A, x_v-x_u\in I_{g(a)}\},$$
and $\ds \chi_{D,U}^g(mp+1)=\l|(mp+1)\De_{D,U}^g\cap \ZZ^V \r|$. Clearly $\ds \De_D^g=\biguplus_{U\subseteq V}\De_{D,U}^g$, hence $\ds \chi_{D}^g(mp+1)=\sum_{U\subseteq V} \chi_{D,U}^g(mp+1)$. It remains to show that there exists a polynomial $P_{g,U}$ such that for any non-negative integer $p$, $\De_{D,U}^g(mp+1)=P_{g,U}(mp+1)$. If $\De_{D,U}^g=\emptyset$, then we can take $P_{g,U}=0$. Hence, we can now assume that $\De_{D,U}^g$ is non-empty.
Since $\De_{D,U}^g$ is non-empty, it is the relative interior of the polytope 
$$\bDe_{D,U}^g=\{(x_v)_{v\in V}\in \RR^{V}~|~\forall v\in U, x_v=1, \forall v\in V\setminus U, x_v\in [0,1],\textrm{ and }\forall a=(u,v)\in A, x_v-x_u\in \bI_{g(a)}\},$$
where $\bI_0=\{0\}$ and for all $k\in [m]$, $\bI_{k}=[(k-1)/m,k/m]$ and $\bI_{-k}=[-k/m,-(k-1)/m]$. 

We now want to conclude using Ehrhart theory, and need to show that the vertices of $\bDe_{D,U}^g$ are in $\frac{1}{m}\ZZ^V$. Any vertex $(x_v)_{v\in V}$ of $\bDe_{D,U}^g$ is the intersection of $|V|$ hyperplanes $h_1,\ldots,h_{|V|}$ of the form $H_{(u,v),c}:=\{x_v-x_u=c\}$, or $H_{v,c}:=\{x_v=c\}$ with $c\in \frac{1}{m}\ZZ$. 
Is is well-known that the incidence matrix of any graph is totally unimodular, so the determinant of the system of linear equations given by $h_1,\ldots,h_{|V|}$ is either $1$ or $-1$, and its solution $(x_v)_{v\in V}$ is in $\frac{1}{m}\ZZ^{V}$. 
Consequently, by Lemma~\ref{lem:Ehrhart-quasi} there is a polynomial $E_{m-1}$ such that for any non-negative integer $p$,
$$\chi_{D,U}^g(mp+1)=\l|(mp+1){\bDe_{D,U}^g}^\bu\cap \ZZ^V \r|=(-1)^{\dim\l(\bDe_{D,U}^g\r)}E_{m-1}(-mp-1).$$
So we can take $P_{g,U}(q)=(-1)^{\dim\l(\De_{D,U}^g\r)}E_{m-1}(-q)$, which concludes the proof.
\end{proof}

Observe that for any positive integer $m$, $\Q^{(m)}(q;y,\ldots,y;z,\ldots,z)=B_D(q,y,z)$. Hence 
the invariant $\Q^{(m)}_D$ refines the $B$-polynomial.
In particular, the polynomial $\Q^{(m)}_D$ can be specialized to the Potts polynomial of the underlying graph $\uD$:
\begin{equation}\label{eq:Potts3-Bm}
\Q_D^{(m)}(q;y,\ldots,y;y,\ldots,y)=B_D(q,y,y)=P_{\uD}(q,y).
\end{equation}
We will now focus on trivariate specializations of $\Q^{(m)}_D$ which generalize the Potts polynomial of graphs.

\begin{definition}
For a binary word $w=w_1,w_2,\cdots, w_m\in\{-1,1\}^m$, we define
$$\Q_D^{w}(q,y,z)=\Q_D^{(m)}(q,y_1,\ldots,y_m;z_1,\ldots,z_m)_{\l|\begin{array}{ll} y_k=y \textrm{ and } z_k=z & \textrm{ if } w_k=1\\y_k=z \textrm{ and }z_k=y & \textrm{ if } w_k=-1 \end{array}
\r.}. $$
Equivalently, $\Q_D^{w}(q,y,z)$ is the unique polynomial such that for all positive integers $p$
$$\Q_D^{w}(mp+1,y,z)=\sum_{f:V\to [mp+1]}y^{\l|f_A^{\As w}\r|}z^{\l|f_A^{\Ds w}\r|},$$
where $f_A^{\As w}$ (resp $f_A^{\Ds w}$) is the set of arcs $(u,v)\in A$ such that $\frac{f(v)-f(u)}{p}$ (resp. $\frac{f(u)-f(v)}{p}$) is in the set 
\begin{equation}\label{eq:Iw} 
I(w):=\biguplus_{k=1}^m w_k\cdot ((k-1),k],
\end{equation}
where $-((k-1),k]$ means $\l[-k,-(k-1)\r)$.
\end{definition}

With the above notation, we have $\Q_D^{1}(q,y,z)=\Q_D^{(1)}(q,y,z)=B_D(q,y,z)$. Observe that the invariants $\Q_D^{w}(q,y,z)$ are not all distinct, since for instance $\Q_D^{1,1}(q,y,z)=\Q_D^{1}(q,y,z)=\Q_D^{-1}(q,y,z)$. We now claim that the relations stated in Theorem~\ref{thm:Bpoly-Tutte} between the $B$-polynomial and the Tutte polynomial hold more generally for $\Q^{w}$-polynomials.


\begin{thm} \label{thm:PottsOneTwo-gle}
Let $w=w_1w_2\cdots w_m\in\{-1,1\}^m$. For any graph $G$, 
\begin{equation}\label{eq:PottsOne-gle}
\Q_{\orient{G}}^{w}(q,y,z)=P_G(q,yz),
\end{equation}
where $\orient{G}$ is the digraph corresponding to $G$ and 
\begin{equation}\label{eq:PottsTwo-gle}
\frac{1}{2^{|E|}}\sum_{\vec{G}\in \ori(G)} \Q^w_{\vec{G}}(q,y,z)
= \Potts_G\left(q,\frac{y+z}{2}\right),
\end{equation}
where the sum is over all digraphs $\vec{G}$ obtained by orienting $G$.
Moreover, for any digraph $D$, 
\begin{equation}\label{eq:PottsThree-gle}
\Q_D^{w}(q,y,y)=P_{\und{D}}(q,y),
\end{equation}
where $\und{D}$ is the graph underlying $D$.
\end{thm}

\begin{proof}
The proofs of~\eqref{eq:PottsOne} and~\eqref{eq:PottsTwo} extend almost verbatim to prove~\eqref{eq:PottsOne-gle} and~\eqref{eq:PottsTwo-gle}. For the proof of~\eqref{eq:PottsOne-gle} we use the fact that $y_kz_k=yz$ for all $k\in[m]$, whereas for the proof of~\eqref{eq:PottsTwo-gle}, we use the fact that $y_k+z_k=y+z$ for all $k\in[m]$.
Finally,~\eqref{eq:PottsThree-gle} is just a reformulation of~\eqref{eq:Potts3-Bm}.
\end{proof}

Equation~\eqref{eq:PottsOne-gle} shows that although the invariants $\Q_D^{w}(q,y,z)$ are not all equal, they all coincide when considering only undirected graphs, and are legitimate generalizations of the Potts polynomial. We now explore the other properties of the $B$-polynomial which extend to the $\Q^{w}$-polynomials. First note that the Properties~\eqref{it1}-\eqref{it8} of Proposition~\ref{prop:easy} extend to $\Q^{w}$, up to changing~\eqref{it3} and~\eqref{it4} into
\begin{compactitem}
\item[(\ref{it3}')] $\ds \sum_{f: V\to [mp+1]}x^{\l|f_A^=\r|}y^{\l|f_A^{\As w}\r|}z^{\l|f_A^{\Ds w}\r|}=x^{|A|}\Q^w_D\l(q,\frac{y}{x},\frac{z}{x}\r)$,
\item[(\ref{it4}')] $\ds \sum_{f: V\to [mp+1]}y^{\l|f_A^{\wAs w}\r|}z^{\l|f_A^{\wDs w}\r|}=(yz)^{|A|}\Q^w_D\l(q,\frac{1}{y},\frac{1}{z}\r)$,
\end{compactitem}
where $f_A^{\wAs w}$ (resp. $f_A^{\wDs w}$) is the set of arcs $(u,v)\in A$ such that $\frac{f(v)-f(u)}{p}$ (resp. $\frac{f(v)-f(u)}{p}$)
is in the set $\ds \{0\}\cup I(w)$. Also the recurrence relations of the $B$-polynomial given in Lemmas~\ref{lem:rec-edge} and~\ref{lem:rec-arc} extend verbatim to the $\Q^w$-polynomials.

We now define the analogues of the strict and weak-chromatic polynomials, and generalize Theorem~\ref{thm:expansions} to the invariants $\Q^w$.
\begin{defn}
Let $D=(V,A)$ be a digraph and let $w=w_1w_2\cdots w_m\in\{-1,1\}^m$. The \emph{$w$-strict-chromatic polynomial} of $D$ is the unique polynomial $\chi^{\As w}_G(q)$ such that for all non-negative integers~$p$,
\begin{equation*}
\chrom_D^{\As w}(mp+1) = \l|\{f:V\to [mp+1]~|~ \forall (u,v)\in A,~ \frac{f(v)-f(u)}{p}\in I(w) \}\r|,
\end{equation*}
where $I(w)$ is defined by~\eqref{eq:Iw}. 
The \emph{$w$-weak-chromatic polynomial} of $D$ is the unique polynomial $\chi^{\wAs w}_G(q)$ such that for all non-negative integers $p$,
\begin{equation*}
\chrom_D^{\wAs w}(mp+1) = \l|\{f:V\to [mp+1]~|~ \forall (u,v)\in A,~ \frac{f(v)-f(u)}{p}\in \{0\}\cup I(w)\}\r|.
\end{equation*}
\end{defn}
Note that $\chi^{\As w}_G(q)=[y^{|A|}]\Q^w_D(mp+1,y,1)$ and $\chrom_D^{\wAs w}(q)=B_D^w(q,0,1)$. Moreover, $\chrom_D^{\As 1}=\chrom_D^{\As}$ and $\chrom_D^{\wAs 1}=\chrom_D^{\wAs}$.
We now state the generalization of Theorem~\ref{thm:expansions}.
\begin{thm}\label{thm:expansions-w}
Let $D=(V,A)$ be a digraph and let $w=w_1w_2\cdots w_m\in\{-1,1\}^m$. Then,
\begin{eqnarray}
\sum_{R\uplus S\uplus T= A} y^{|S|}z^{|T|}\chrom_{D^{-T}_{\setminus R}}^{\As w}(q)&=& \Q^w_D(q,1+y,1+z),\label{>Delete-w}\\
\sum_{R\uplus S\uplus T= A} y^{|S|}z^{|T|}\chrom_{D^{-T}_{\setminus R}}^{\wAs w}(q)&=&(1+y+z)^{|A|}\Q^w_D\left(q,\frac{1+y}{1+y+z},\frac{1+z}{1+y+z}\right),\label{geqDelete-w}\\
\sum_{R\uplus S\uplus T = A} y^{|S|}z^{|T|}\chrom_{D^{-T}_{/R}}^{\As w}(q) &=&\Q^w_D(q,y,z),\label{>Contract-w}\\
\sum_{R\uplus S\uplus T = A} y^{|S|}z^{|T|}\chrom_{D^{-T}_{/R}}^{\wAs w}(q)&=&(1+y+z)^{|A|}\Q^w_D\left(q,\frac{y}{1+y+z},\frac{z}{1+y+z}\right).\label{geqContract-w}
\end{eqnarray}
\end{thm}

\begin{proof}
This is a straightforward adaptation of the proof of Theorem~\ref{thm:expansions} which is left to the reader.
\end{proof}

It should be mentioned that other properties of the $B$-polynomial, such as those established in Sections~\ref{sec:Ehrhart} and~\ref{sec:Tutte-eval} (and the quasisymmetric refinement of Section~\ref{sec:quasisym}) do not extend to $\Q^w$-polynomials for arbitrary $w$.

We now focus on a subfamily of the $\Q^w$-polynomials which enjoy additional properties. We say that a word $w=w_1\ldots w_m\in\{-1,1\}^m$ is \emph{antipalindromic} if for all $k\in[m]$, $w_{m+1-k}=-w_k$. We will show that when $w$ is antipalindromic, the invariant $\Q^w_D$ is essentially determined by the oriented matroid underlying $D$.

We start with the key observation. Let $D=(V,A)$ be a digraph, and let $f$ be a $q$-colorings of $D$. We denote by $\ov{f}:A\to \ZZ/q\ZZ$ the function defined by $\ov{f}(u,v)=f(v)-f(u)$ mod $q$, and we call $\ov{f}$ a \emph{$q$-coflow} of $D$. 
\begin{lemma}\label{lem:coflow}
Each $q$-coflow of $D$ corresponds to $q^{\comp(D)}$ distinct $q$-colorings. Moreover, if $w$ is antipalindromic, then for all non-negative integers $p$
$$\Q^w_D(mp+1,y,z)=(mp+1)^{\comp(D)}\sum_{\ov{f}:A\to \ZZ/(mp+1)\ZZ,\textrm{ $(mp+1)$-coflow}}y^{\ov{f}_A^{\As w}}z^{\ov{f}_A^{\Ds w}},$$
where $\ov{f}_A^{\As w}$ (resp. $\ov{f}_A^{\Ds w}$) is the set of arcs $a\in A$ such that $\ov{f}(a)\equiv (k-1)p+r$ mod $(mp+1)$, with $r\in[p]$ and with $k\in[m]$ such that $w_k=1$ (resp. $w_k=-1$).
\end{lemma}

\begin{proof}
First, it is clear that two $q$-colorings $f,g$ satisfy $\ov{f}=\ov{g}$ if and only if if for each component $D'$ of $D$, there is a constant $c$ such that for every vertex $v$ of $D'$, $f(v)\equiv g(v)+c$ modulo $q$. So for each $q$-coloring $f$, there are $q^{\comp(D)}$ $q$-colorings $g$ such that $\ov{g}=\ov{f}$. 

We now suppose that $w$ is antipalindromic. It suffices to prove that any $(mp+1)$-coloring $f$ satisfies $f_A^{\As w}=\ov{f}_A^{\As w}$ and $f_A^{\Ds w}=\ov{f}_A^{\Ds w}$. By definition, $(u,v)\in f_A^{\As w}$ if and only if there exists $k\in [m]$ such that either $w_k=1$ and $(k-1)p+1\leq f(v)-f(u)\leq kp$, or $w_k=-1$ and $-kp\leq f(v)-f(u)\leq -(k-1)p-1$. Since $w$ is antipalindromic, the second case can be rewritten as $w_{m+1-k}=1$ and $(m+1-k)p-(mp+1)\leq f(v)-f(u)\leq (m+1-k)p-(mp+1)$. Thus, $(u,v)\in f_A^{\As w}$ if and only if $(u,v)\in \ov{f}_A^{\As w}$. Similarly, $(u,v)\in f_A^{\Ds w}$ if and only if $(u,v)\in \ov{f}_A^{\Ds w}$. 
\end{proof}

%


 
Recall that each digraph $D$ has an underlying \emph{oriented matroid} $M_D$, which is the oriented matroid whose circuits are given by the simple cycles of $D$. We refer the reader to~\cite{Bjorner:oriented-matroids} for definitions about oriented matroids, but we will not use any result from oriented matroid theory. 

\begin{cor}
If $w\in\{-1,1\}^m$ is antipalindromic, then for any digraphs $D$, the polynomial $\Q^w_D(q,y,z)$ is divisible by $q^{\comp(D)}$. Moreover, the polynomial invariant
 $q^{-\comp(D)}\Q^w_D(q,y,z)$ only depends on the oriented matroid $M_D$ underlying $D$ (that is, $M_D=M_{D'}$ implies $q^{-\comp(D)}\Q^w_D(q,y,z)=q^{-\comp(D')}\Q^w_{D'}(q,y,z)$).
\end{cor}

\begin{proof}
Let $D$ be a connected digraph. Let $a,b$ be non-negative integers, and let $P(q):=[y^az^b]\Q^w_D(q,y,z)$. We know that $P(q)$ is a polynomial in $q$ and we want to show that it is divisible by~$q$.
By Lemma~\ref{lem:coflow}, we know that for any non-negative integer $p$, the value $P(mp+1)$ is an integer divisible by $mp+1$.
Since the polynomial $Q(p)=P(mp+1)$ is integer valued, $Q(p)$ has rational coefficients (by Lagrange interpolation). Hence  $P(q)$ has rational coefficients. Thus, there is a rational number $r$ and a polynomial $R(q)$ with rational coefficients such that $\frac{1}{q}P(q)=r/q+R(q)$. Let $d$ be the least common multiple of the denominators of the coefficients of $R$. Then for all non-negative integers $p$, $r/(mp+1)\in \frac{1}{d}\ZZ$. Taking $p$ large enough shows that $r=0$. Hence, $\frac{1}{q}P(q)$ is a polynomial. This proves that $q$ divides $\Q^w_D(q,y,z)$. Since the invariant $\Q^w$ is multiplicative over connected components, this implies that for any digraph $q^{\comp(D)}$ divides $\Q^w_D(q,y,z)$.

Next, we show that the polynomial $q^{-\comp(D)}\Q^w_D$ only depends on the oriented matroid $M_D$ underlying $D$. By Lemma~\ref{lem:coflow} we only need to show that for any positive integer~$q$, the set of $q$-coflows only depends on $M_D$. Moreover, it is easy to see that a function $g:A\to \ZZ/q\ZZ$ is a $q$-coflow if and only if along any simple cycle $C$ of $D$, 
$$\sum_{a\in C^-}g(a)\equiv\sum_{a\in C^+}g(a),$$
where $C^-$ and $C^+$ are the sets of arc in one direction and the other direction along $D$. This characterization only depends on the simple cycles of $D$, hence only on $M_D$.
\end{proof}

\begin{remark}
By contrast to the case of antipalindromic words $w$, it is clear that the invariant $B_D$ is \emph{not} determined by the underlying oriented matroid $M_D$ and the number of components. Indeed, by Corollary~\ref{cor:additional-prop}, the invariant $B_D$ detects the length of directed paths which cannot be detected from $M_D$ (for instance every oriented tree has the same underlying oriented matroid).
\end{remark}


In the companion paper~\cite{Awan-Bernardi:A-poly} we will investigate in more detail the invariant $q^{-\comp(D)}\Q^{w}_D$ corresponding to the simplest antipalindromic word $w=1,-1$, for digraphs and more generally for regular oriented matroids. The invariant $q^{-\comp(D)}\Q^{1,-1}_D$ is denoted by $A_D$ in that paper. \\

\section{Concluding remarks}\label{sec:conclusion}
We have shown that the $B$-polynomial is a natural generalization of the Potts polynomial (or equivalently, Tutte Polynomial) to digraphs.
We have explored several aspects of this invariant. However, this leaves many more open questions, besides the two combinatorial puzzles already mentioned in Questions~\ref{question:mysterious-identities-1} and~\ref{question:mysterious-identities}.

One subject we did not attempt to cover is the computational complexity of determining the $B$-polynomial of a given digraph $D$. 
Given that $B_D(q,y,z)$ is a polynomial in $q$ of degree $|V|$, it is uniquely determined by the values $\{B_D(q,y,z)\}$ for all $q\in[|V|]$, and the fact that $B_D(0,y,z)=0$. Hence, using a naive approach gives a way of determining $B_D(q,y,z)$ in $O(|A|\cdot|V|^{|V|})$ operations. Using~\eqref{eq:q-falling-factorial} actually gives a slightly better complexity of $\ln(2)^{-|V|}|V|!$ up to a subexponential factor (and~\eqref{eq:monomial-basis} gives the same complexity for computing the quasisymmetric $B$-polynomial). We do not know if a more efficient method exists.

\begin{question}
What is the complexity of computing $B_D(q,y,z)$ in terms of $|V|$? What are the specializations of $B_D(q,y,z)$ which are computable in polynomial time?
\end{question}

In Section~\ref{sec:Tutte-eval}, we gave an interpretation of $T^{(1)}_D(2,0)$, $T^{(2)}_D(2,0)$, and $T^{(2)}_D(0,2)$ as counting some classes of orientations of the mixed graph $D$. In view of the known results for the Tutte polynomial of graphs, we ask the following question.

\begin{question} 
For $i\in\{1,2\}$, is there an interpretations for the evaluations $T_D^{(i)}(1,1)$, $T_D^{(i)}(0,1)$, $T_D^{(i)}(1,0)$, $T_D^{(i)}(2,1)$, and $T_D^{(i)}(1,2)$ as counting some classes of orientations of a mixed graph $D$?
\end{question}


Recall from Proposition~\ref{prop:easy} that $B_D(q,y,z)$ is symmetric in $y$ and $z$. Therefore $B_D(q,y,z)$ is a polynomial in $q$, $yz$ and $y+z$. Moreover, when $D$ corresponds to a graph (that is, $D=\orient{G}$ for some $G$), then~\eqref{eq:PottsOne} shows that $B_D(q,y,z)$ is actually a polynomial in $q$ and $yz$ only. It would be interesting to know if the following converse is true.

\begin{question}
Is it true that $B_D(q,y,z)$ is a function of $q$ and $yz$ if and only if $D=\orient{G}$ for a graph $G$?
\end{question}

For planar digraphs, we have established a duality relation~\eqref{eq:duality} for the specialization $B_D(-1,y,z)$ of the $B$-polynomial, while the classical duality relation for the Tutte polynomial gives a duality relation~\eqref{eq:duality-Bclassical} for the specialization $B_D(q,y,y)$. We have not found a common generalization of these two relations in general.

\begin{question}
Is there an infinite class of planar digraphs for which there exists a common generalization of the two duality relations~\eqref{eq:duality} and~\eqref{eq:duality-Bclassical}? 
\end{question}

Let $D$ a digraph. Recall from Remark~\ref{rk:PottsThreedetects} that the number of spanning trees of the underlying graph $\uD$ can be obtained from $B_D(q,y,z)$. In contrast, Remark~\ref{rk:not-detected} shows that the number of \emph{directed spanning trees} of $D$ (spanning trees oriented in such a way that every vertex except one has indegree 1) cannot be obtained from $B_D(q,y,z)$. 

\begin{question}
Is it possible to obtain from $B_D(q,y,z)$ the number of \emph{alternating spanning trees} of $D$, that is, spanning trees $T$ of $D$ such that every vertex of $T$ is either a source or a sink of $T$?
Is it possible to obtain from $B_D(\xx;y,z)$ either the number of \emph{directed spanning trees} of $D$ or the number of \emph{alternating spanning trees} of $D$?
\end{question}

Lastly, one can wonder which classes of digraphs are distinguished by $B_D(q,y,z)$, and by $B_D(\xx;y,z)$. We say that a digraph invariant \emph{distinguishes between} a class $C$ of digraphs if the digraphs in $C$ can be determined from the value of the invariant (equivalently, there does not exist non-isomorphic digraphs in $C$ with the same value of the invariant). It is clear that $B_D(\xx;y,z)$ does not distinguish between all digraphs, and not even between all the digraphs corresponding to graphs. Indeed, by~\eqref{eq:PottsOne-multi} for all graphs $G$, $B_{\orient{G}}(\xx;y,z)$ is equivalent to Tutte symmetric function $S_G(\xx;y)$, and this invariant does not distinguish between all graphs~\cite{Stanley:symmetric-chromatic-more}. However, in~\cite{Stanley:symmetric-chromatic} Stanley raised the question of whether the chromatic symmetric function can distinguish between all trees. 
This question has generated a lot of interest, but is still open~\cite{Zamora:distinguishing-caterpillars,Loebl:distinguishing-trees,Zamora:distinguishing-trees,Martin:distinguishing-trees}. Let us also mention the article~\cite{McNamara:equality-Ppartitions}, in which some necessary conditions are given for two acyclic digraphs to have the same weak-chromatic polynomial (and more generally, conditions for labeled posets to have the same quasisymmetric functions). 
Similar questions are natural here.
\begin{question}\label{question:distinguishing}
Does the quasisymmetric $B$-polynomial distinguish between
\begin{compactenum}
\item[(i)] all \emph{oriented trees} (digraph whose underlying graph is a tree)? 
\item[(ii)] all \emph{directed trees} (oriented trees such that every vertex except one has indegree 1)? 
\item[(iii)] all \emph{alternating trees} (oriented trees such that every vertex is either a source or a sink)?
\item[(iv)] all \emph{alternating digraphs} (digraphs such that every vertex is either a source or a sink)? 
\end{compactenum}
\end{question}

Answering either cases (i), (ii) or (iii) of Question~\ref{question:distinguishing} would hopefully shed some light on Stanley's original question. Questions (ii) and (iii) seems intuitively more tractable than Stanley's original question. For instance, it is clear from Remark~\ref{rk:detect-quasi} that the quasisymmetric $B$-polynomial distinguishes between the directed caterpillars whose root-vertex (vertex of indegree 0) is at an extremity of the spine (because such caterpillars are distinguished by their profile). 
However, to our shame, we do not know if a positive answer to Stanley's original question would give a positive answer to cases (ii) or (iii).
Indeed we do not know the answer to the following seemingly easier questions.
\begin{question}
\begin{compactenum}
\item[(i)] Let $G$ be any fixed graph. Does the quasisymmetric $B$-polynomial distinguish between all the non-isomorphic digraphs obtained by orienting $G$?
\item[(ii)] Let $T$ be any fixed tree. Does the quasisymmetric $B$-polynomial distinguish between all the non-isomorphic oriented (resp. directed, alternating) trees obtained by orienting $T$?
\end{compactenum}
\end{question}

Finally, let us mention that the motivation for the case (iv) of Question~\ref{question:distinguishing} is that the class of alternating digraphs identifies with the class of \emph{hypergraphs}. By \emph{hypergraph} we mean a bipartite graph $G=(V,E)$ with a prescribed bipartition of the vertex set $V=V_1\uplus V_2$. The identification with alternating digraphs is simply obtained by orienting every edge from $V_1$ to $V_2$. Because of this correspondence, it would be interesting to study what properties of alternating digraphs can be obtained from $B_D(q,y,z)$, and from $B_D(\xx;y,z)$.\\

\bigskip

\noindent \textbf{Acknowledgments:} 
We are grateful to Ira Gessel and Darij Grinberg for very valuable discussions about quasi-symmetric functions, and to Mathias Beck, Valentin Feray, Martin Loebl, Lorenzo Traldi, and Dominic Welsh for bibliographical references, and to Sam Hopkins for his explanations about the relevance of fourientations activities to Question~\ref{question:mysterious-identities}. We also thank the referees for suggesting improvements to the presentation of this paper.

\bibliographystyle{plain} 
\bibliography{biblio-Tutte}

\end{document}